\documentclass[11pt,leqno]{amsart}
\usepackage{verbatim,amssymb,amsfonts,latexsym,amsmath,amsthm,tikz,xspace,pgfplots,mathtools}

\newtheorem{theorem}{Theorem}[section]
\newtheorem{lemma}[theorem]{Lemma}
\newtheorem{corollary}[theorem]{Corollary}

\newtheorem{question}[theorem]{Question}
\newtheorem*{maintheorem}{Main Theorem}
\newtheorem*{theorem*}{Theorem}
\newtheorem*{corollary*}{Corollary}
\newtheorem*{observation}{Observation}

\newtheorem*{sub-claim}{sub-claim}
\theoremstyle{definition}
\newtheorem{example}[theorem]{Example}
\newtheorem{definition}[theorem]{Definition}

\theoremstyle{remark}
\newtheorem{remark}[theorem]{Remark}
\newtheorem*{definition*}{Definition}

\newcommand{\R}{\mathbb{R}}
\newcommand{\N}{\mathbb{N}}

\newcommand{\A}{\mathcal{A}}
\newcommand{\B}{\mathcal{B}}
\newcommand{\C}{\mathcal{C}}

\newcommand{\U}{\mathcal{U}}

\newcommand{\explicitSet}[1]{\left\lbrace #1 \right\rbrace}
\newcommand{\brackets}[1]{\left\langle #1 \right\rangle}
\newcommand{\set}[2]{\explicitSet{#1 \colon #2}}
\newcommand{\seq}[2]{\brackets{#1 \colon #2}}

\renewcommand{\a}{\alpha}
\renewcommand{\b}{\beta}
\newcommand{\g}{\gamma}
\newcommand{\dlt}{\delta}
\newcommand{\e}{\varepsilon}
\newcommand{\z}{\zeta}
\renewcommand{\k}{\kappa}
\newcommand{\s}{\sigma}
\renewcommand{\t}{\tau}
\newcommand{\w}{\omega}
\newcommand{\0}{\emptyset}
\newcommand{\sub}{\subseteq}
\newcommand{\rest}{\!\restriction\!}
\newcommand{\domain}{\mathrm{dom}}
\newcommand{\cat}{\!\,^{{}_{{}^\frown\!}}}

\newcommand{\closure}[1]{\overline{#1}}

\newcommand{\cf}{\mathrm{cf}}
\newcommand{\card}[1]{\left\lvert #1 \right\rvert}
\newcommand{\PP}{\mathbb{P}}

\newcommand{\HH}{\mathbb{H}}

\newcommand{\BB}{\mathbb{B}}
\newcommand{\DD}{\mathbb{D}}
\newcommand{\continuum}{\mathfrak{c}}
\newcommand{\pseudo}{\mathfrak{p}}

\newcommand{\dom}{\mathfrak d}
\newcommand{\bdd}{\mathfrak b}

\newcommand{\ch}{\ensuremath{\mathsf{CH}}\xspace}
\newcommand{\zfc}{\ensuremath{\mathsf{ZFC}}\xspace}

\newcommand{\ma}{\ensuremath{\mathsf{MA}}\xspace}

\newcommand{\VL}{V=\mathrm{L}}
\newcommand{\gch}{\ensuremath{\mathsf{GCH}}\xspace}
\newcommand{\vws}{\ensuremath{\mathsf{VWS}}\xspace}

\newcommand{\axiom}{\raisebox{.5mm}{$\bigtriangledown$}\xspace}
\newcommand{\nt}{\mathrm{Nt}}
\newcommand{\pnt}{\pi\mathrm{Nt}}
\newcommand{\ro}{\mathsf{ro}}
\newcommand{\down}{\hspace{-.5mm}\downarrow\hspace{.5mm}}

\begin{document}

\title{Telg\'arsky's conjecture may fail}
\author{Will Brian}
\address {
Will Brian\\
Department of Mathematics and Statistics\\
9201 University City Blvd.\\
Charlotte, NC 28223}
\email{wbrian.math@gmail.com}
\urladdr{wrbrian.wordpress.com}
\author{Alan Dow}
\address {
Alan Dow\\
Department of Mathematics and Statistics\\
University of North Carolina at Charlotte\\
Charlotte, NC 28223}
\email{adow@uncc.edu}
\author{David Milovich}
\address {
David Milovich\\
Welkin Sciences\\
Colorado Springs, CO 80903}
\email{david.milovich@welkinsciences.com}
\urladdr{dkmj.org}
\author{Lynne Yengulalp}
\address {
Lynne Yengulalp\\
Department of Mathematics\\
University of Dayton\\
Dayton, OH 45469}
\email{lyengulalp1@udayton.edu}
\subjclass[2010]{91A44, 03E05, 03E35}
\keywords{Banach-Mazur game, $k$-tactics, posets, constructible universe}

\begin{abstract}
Telg\'arsky's conjecture states that for each $k \in \mathbb N$, there is a topological space $X_k$ such that in the Banach-Mazur game on $X_k$, the player {\scriptsize NONEMPTY} has a winning $(k+1)$-tactic but no winning $k$-tactic. 
We prove that this statement is consistently false.

More specifically, we prove, assuming $\mathsf{GCH}+\square$, that  if {\scriptsize NONEMPTY} has a winning strategy for the Banach-Mazur game on a $T_3$ space $X$, then she has a winning $2$-tactic. The proof uses a coding argument due to Galvin, whereby if $X$ has a $\pi$-base with certain nice properties, then {\scriptsize NONEMPTY} is able to encode, in each consecutive pair of her opponent's moves, all essential information about the play of the game before the current move. Our proof shows that under $\mathsf{GCH}+\square$, every $T_3$ space has a sufficiently nice $\pi$-base that enables this coding strategy.

Translated into the language of partially ordered sets, what we really show is that $\mathsf{GCH}+\square$ implies the following statement, which is equivalent to the existence of the ``nice'' $\pi$-bases mentioned above:
\begin{itemize}
\item[\raisebox{.5mm}{$\bigtriangledown$}:] Every separative poset $\mathbb P$ with the $\kappa$-cc contains a dense sub-poset $\mathbb D$ such that $|\{ q \in \mathbb D \,:\, p \text{ extends } q \}| < \kappa$ for every $p \in \mathbb P$. 
\end{itemize}
We prove that this statement is independent of $\mathsf{ZFC}$: while it holds under $\mathsf{GCH}+\square$, it is false even for ccc posets if $\mathfrak{b} > \aleph_1$. We also show that if $|\mathbb P| < \aleph_\omega$, then \axiom-for-$\mathbb P$ is a consequence of $\mathsf{GCH}$ holding below $|\mathbb P|$.
\end{abstract}

\maketitle

\section{Introduction}

The Banach-Mazur game appeared in 1935, in question 43 of the Scottish Book, the now-famous notebook of Stefan Banach begun earlier the same year.\footnote{A scan of Mazur's question, in its original form, can be seen in \cite{ScottishBook1}, or an English translation in \cite{ScottishBook2}.} The author of the question was Stanis\l aw Mazur, and a solution was found by Banach later in 1935. Banach showed (in our terminology) that in the Banach-Mazur game on some $X \sub \R$, player {\small NONEMPTY} has a winning strategy if and only if $X$ is co-meager, and player {\small EMPTY} has a winning strategy if and only if $X$ is meager on some non-degenerate interval. (A proof can be found, e.g., in \cite[Chapter 6]{Oxtoby}.) The Banach-Mazur game is the first infinite game of perfect information to be studied. 

A \emph{$k$-tactic} in the Banach-Mazur game is a strategy for one player that depends only on the previous $k$ moves of the opposing player. This is one example of a \emph{limited-information strategy}, a recurring theme in the study of the Banach-Mazur game and other topological games. Such strategies were studied by Debs in \cite{Debs}, where he proved that there is a space $X$ for which {\small NONEMPTY} has a winning $2$-tactic, but no winning $1$-tactic. Shortly after Debs' paper appeared, Telg\'arsky conjectured in \cite[page 236]{Telgarsky} that for every $k \geq 2$, there is a space $X_k$ such that {\small NONEMPTY} has a winning $(k+1)$-tactic in the Banach-Mazur game on $X_k$, but no winning $k$-tactic. (See also problems 204-206 in \cite{OPIT}, and Conjecture 2 in \cite{BJS}.) Our main theorem in this paper shows that this conjecture, when restricted to $T_3$ spaces, is not provable from \zfc.

\begin{maintheorem}
Assume $\mathsf{GCH}+\square$. For every $T_3$ space $X$, if \emph{\small NONEMPTY} has a winning strategy in the Banach-Mazur game on $X$, then \emph{\small NONEMPTY} has a winning $2$-tactic in  the Banach-Mazur game on $X$.
\end{maintheorem}

\noindent This is mildly abridged version of the main theorem: the full version (Theorem~\ref{thm:main} below) has a weaker hypothesis than $\square$, and is stated for the class of quasi-regular spaces, which is broader than the class of $T_3$ spaces. 

The proof of this theorem uses a coding argument due to Galvin, whereby if $X$ has a $\pi$-base with certain nice properties, then {\small NONEMPTY} is able to encode, in each consecutive pair of her opponent's moves, all essential information about the play of the game before the current round.

In Theorem~\ref{thm:translation} below, we will see that the existence of these sufficiently nice $\pi$-bases for quasi-regular spaces is equivalent to the following statement concerning partially ordered sets: 
\begin{itemize}
\item[$\axiom$:] Every separative poset $\mathbb P$ with the $\kappa$-cc contains a dense sub-poset $\mathbb D$ such that $|\{ q \in \mathbb D \,:\, p \text{ extends } q \}| < \kappa$ for every $p \in \mathbb P$. 
\end{itemize}
In short, our proof works by showing this statement is consistent, and this suffices to prove the main theorem via Galvin's coding argument. 

The proof of the consistency of \axiom uses a generalization to higher cardinals of the combinatorial structures called ``sage Davies trees'' in \cite{Soukup^2}. The existence of the sage Davies trees of \cite{Soukup^2} suffices to prove \axiom for ccc partial orders; the generalized structures are introduced to handle $\k$-cc partial orders for uncountable $\k$. Our construction of these structures uses \gch plus a very weak version of $\square$. If $|\mathbb P| < \aleph_\omega$, then \axiom-for-$\mathbb P$ is a consequence of \gch holding below $|\mathbb P|$. 

We also show the independence of \axiom from \zfc by proving that if $\bdd > \aleph_1$, then the Hechler forcing does not satisfy \axiom. Similarly, $\ma+\neg\ch$ implies that both the random real forcing and a ccc variant of Mathias forcing fail to satisfy \axiom. 

The topological content of the paper is contained in Section~\ref{sec:BM}. There we review some facts concerning the Banach-Mazur game, exposit Galvin's (previously unpublished) coding argument, and show the hypotheses of Galvin's theorem are equivalent to \axiom.
The relative consistency of \axiom is proved in Section~\ref{sec:V=L} via the construction of generalized Davies trees.
Section~\ref{sec:independence} contains the independence results mentioned in the previous paragraph.

\section{The Banach-Mazur game}\label{sec:BM}

Let $X$ be a nonempty topological space.
The Banach-Mazur game on $X$, denoted $\mathrm{BM}(X)$, is played by two players, whom we call {\small EMPTY} and {\small NONEMPTY}, alternately choosing nonempty open subsets of $X$ as follows. In round $0$ of the game, {\small EMPTY} chooses any nonempty open $U_0 \sub X$, and then {\small NONEMPTY} chooses any nonempty open $V_0 \sub U_0$. In round $1$, {\small EMPTY} chooses a nonempty open $U_1 \sub V_0$, and then {\small NONEMPTY} chooses a nonempty open $V_1 \sub U_1$. Continuing in this way, the players select an infinite sequence 
$U_0 \supseteq V_0 \supseteq U_1 \supseteq V_1  \supseteq U_2 \supseteq V_2 \supseteq \dots$
of open subsets of $X$. {\small NONEMPTY} wins this play of the game provided that $\bigcap_{n \in \w}U_n = \bigcap_{n \in \w}V_n \neq \0$, and otherwise {\small EMPTY} wins.

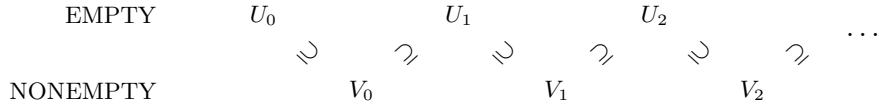
\vspace{2mm}
\begin{center}
\begin{tikzpicture}

\node at (0.37,1) {\footnotesize EMPTY};
\node at (0,0) {\footnotesize NONEMPTY};

\node at (2.4,1) {\footnotesize $U_0$};
\node at (3.7,0) {\footnotesize $V_0$};
\begin{scope}[shift={(3,.5)}]
\node[rotate=-42] at (0,0) {\footnotesize $\supseteq$};
\end{scope}
\begin{scope}[shift={(4.3,.5)}]
\node[rotate=42] at (0,0) {\footnotesize $\supseteq$};
\end{scope}

\node at (5,1) {\footnotesize $U_1$};
\node at (6.3,0) {\footnotesize $V_1$};
\begin{scope}[shift={(5.6,.5)}]
\node[rotate=-42] at (0,0) {\footnotesize $\supseteq$};
\end{scope}
\begin{scope}[shift={(6.9,.5)}]
\node[rotate=42] at (0,0) {\footnotesize $\supseteq$};
\end{scope}

\node at (7.6,1) {\footnotesize $U_2$};
\node at (8.9,0) {\footnotesize $V_2$};
\begin{scope}[shift={(8.2,.5)}]
\node[rotate=-42] at (0,0) {\footnotesize $\supseteq$};
\end{scope}
\begin{scope}[shift={(9.5,.5)}]
\node[rotate=42] at (0,0) {\footnotesize $\supseteq$};
\end{scope}

\node at (10.4,.75) {$\dots$};

\end{tikzpicture}
\end{center}

\begin{definition}
A \emph{strategy} for a player in $\mathrm{BM}(X)$ is a rule for choosing what to play in any given round, given all the preceding plays. 
Formally, a strategy for {\small NONEMPTY} is a function $\s$ mapping each nested sequence of nonempty open sets $U_0 \supseteq V_0 \supseteq \dots \supseteq U_{n-1} \supseteq V_{n-1} \supseteq U_n$ to some nonempty open $V_n \sub U_n$. (And a strategy for {\small EMPTY} is defined analogously.) A \emph{winning strategy} for a given player is a strategy that always produces a win for that player.
\end{definition}


For example, if $X$ is a compact Hausdorff space, then a winning strategy for {\small NONEMPTY} could be: given $U_n$, choose any nonempty open $V_n$ such that $\closure{V_n} \sub U_n$. Note that this strategy, when applied in round $n$, ignores all of the gameplay from previous rounds, and only takes into account the play of {\small EMPTY} from the first part of round $n$. This example exhibits what is called a \emph{stationary} winning strategy for {\small NONEMPTY}, which means that the strategy only depends on the previous move of her opponent. This is one example of what is called a \emph{limited information strategy}. Other important kinds of limited information strategies include ($k$-)Markov strategies \cite{Galvin&Telgarsky} and coding strategies \cite{Debs,Galvin&Telgarsky}. For the remainder of this section, we will focus on the following generalization of stationary strategies.

\begin{definition}
A \emph{$k$-tactic} for {\small NONEMPTY} in $\mathrm{BM}(X)$ is strategy that depends only on the previous $k$ moves of {\small EMPTY}. That is, $\s$ is a $k$-tactic if and only if there is a function $\varsigma$, defined on $k$-length sequences of open sets, such that $\s(U_0,V_0,\dots,U_{n-1},V_{n-1},U_n) = \varsigma(U_{n-k+1},\dots,U_{n-1},U_n)$ for every sequence $U_0 \supseteq V_0 \supseteq \dots \supseteq U_{n-1} \supseteq V_{n-1} \supseteq U_n$.
\end{definition}

Thus, for example, a stationary strategy is the same thing as a $1$-tactic. Our interest in $k$-tactics begins with the following theorem of Debs:

\begin{theorem*}[Debs \cite{Debs}, 1985]
There is a completely regular space $X$ for which \emph{{\small NONEMPTY}} has a winning $2$-tactic but no winning $1$-tactic. 
\end{theorem*}

It is fairly easy to show that {\small NONEMPTY} has a winning (full information) strategy in the Banach-Mazur game on Debs' space. The existence of a winning $2$-tactic is more difficult to prove. The key is for {\small NONEMPTY} to use a topological property of Debs' space to reduce an arbitrary full-information strategy to a $2$-tactic. In this $2$-tactic, {\small NONEMPTY} does not really have to rely on limited information: instead, the entire history of the game (slightly modified) is coded into {\small EMPTY}'s two previous moves, so that {\small NONEMPTY} simply has to decode it, and then play according to her full-information winning strategy.
Before describing this coding strategy in detail, we review some topological terminology.

\begin{definition}
Let $X$ be a topological space. 
A \emph{cellular family} in $X$ is a collection of nonempty pairwise disjoint open subsets of $X$. A cellular family $\mathcal S$ is \emph{maximal} if it is not properly contained in any other cellular family, or, equivalently, if $\bigcup \mathcal S$ is dense in $X$.
The \emph{Souslin number} of $X$, denoted $S(X)$, is defined as
$$S(X) = \min \set{\k}{X \text{ has no cellular family of size } \k}$$
and $X$ is called \emph{ccc} if $S(X) \leq \aleph_1$.
A \emph{$\pi$-base} for $X$ is a collection $\B$ of nonempty open subsets of $X$ such that every nonempty open subset of $X$ contains a member of $\B$. The \emph{Noetherian type} of $\B$, denoted $\nt(\B)$, is  
$$\nt(\B) = \min \set{\k}{ \text{for all nonempty open } U \sub X, \card{\set{V \in \B}{U \sub V}} < \k}.$$
The \emph{$\pi$-Noetherian type} of $X$ is
$$\pnt(X) = \min \set{\nt(\B)}{\B \text{ is a } \pi \text{-base for } \k}.$$
\end{definition}

\vspace{2mm}

The following theorem is an unpublished result of Galvin, also recorded (without proof) as Theorem 39 in \cite{BJS}. The idea of the theorem is just to extend to a general setting Debs' coding idea that converts an arbitrary winning strategy into a winning $2$-tactic. We record a proof of it here, as Galvin's theorem is the link between our main theorem and the set-theoretic results of the next section.

\begin{theorem}[Galvin]\label{thm:galvin}
Let $X$ be a space for which \emph{\small NONEMPTY} has a winning strategy in $\mathrm{BM}(X)$. If $\pnt(U) \leq S(U)$ for all nonempty open $U \sub X$, then \emph{\small NONEMPTY} has a winning $2$-tactic in $\mathrm{BM}(X)$.
\end{theorem}

\begin{proof}
To begin, first observe that if $U,V \sub X$ are open, then $V \sub U$ implies $S(V) \leq S(U)$. It follows that for every nonempty open $U \sub X$, there is some nonempty $V \sub U$ such that $S(V) = S(W)$ for all nonempty open $W \sub V$. (Otherwise we could find a decreasing sequence $U \supseteq V_1 \supseteq V_2 \supseteq \dots$ such that $S(V_{n+1}) < S(V_n)$ for all $n$, which is absurd.)

Using this fact, a straightforward application of Zorn's Lemma shows that there is a maximal cellular family $\C$ such that every $U \in \C$ has the property that $S(U) = S(V)$ for all nonempty open $V \sub U$. Fix some such $\C$. For every $U \in \C$, let $\B_U$ be a $\pi$-base for $U$ witnessing the inequality $\pnt(U) \leq S(U)$; in other words, $\B_U$ is a $\pi$-base for $U$ such that 
$$\card{\set{W \in \B_U}{V \sub W}} < S(U) = S(V) \text{ for all nonempty open } V \sub U.$$
Let $\B = \bigcup \set{\B_U}{U \in \C}$. Because $\C$ is maximal, $\B$ is a $\pi$-base for $X$. Because each nonempty $V \sub X$ is contained in at most one $U \in \C$,
\begin{equation}
\card{\set{W \in \B}{V \sub W}} < S(V) \text{ for all nonempty open } V \sub X.
\tag{$*$}
\end{equation}

Note that for every $V \in \B$, either $S(V)$ is infinite or $S(V) = 2$. (The case $S(V) = 2$, meaning that $V$ contains no two disjoint nonempty open sets, occurs in Hausdorff spaces if and only if $V$ consists of a single isolated point; but in non-Hausdorff spaces, there are other ways this can happen.) 
To see this, note that if $S(V) = n > 2$ and $\C = \{W_1,\dots,W_{n-1}\}$ is a cellular family in $V$, then each $W_i$ must have $S(V) = 2$, since otherwise we could replace $W_i$ in $\C$ with $\geq\! 2$ disjoint nonempty open subsets of $W_i$ to obtain a cellular family in $V$ of size $\geq\! n$, contradicting $S(V) = n$. Thus $S(W_i) = 2 < S(V)$ and by construction, no such $V$ is in $\B$.

For each $V \in \B$, fix an injective function $\psi_V: [\set{W \in \B}{V \sub W}]^{<\w} \to \B$ such that the image of $\psi_V$ is a cellular family in $V$. The existence of such an injection follows from the previous paragraph together with property $(*)$, which together imply $\card{\vphantom{f^{f^f}}[\set{W \in \B}{V \sub W}]^{<\w}} < S(V)$. 
Also fix a function $\pi$ from the collection of all nonempty open subsets of $X$ into $\B$ with the property that $\pi(U) \sub U$ for all $U$.

Suppose $\s$ is a winning strategy for {\small NONEMPTY} in $\mathrm{BM}(X)$. We now construct a winning $2$-tactic $\varsigma$ for {\small NONEMPTY} by describing how {\small NONEMPTY} should respond to any possible sequence of plays in $\mathrm{BM}(X)$. 

To begin the game, {\small EMPTY} plays some nonempty open $U_0 \sub X$. Let $\widehat U_0 = \psi_{\pi(U_0)}(\{\pi(U_0)\})$, and define $\varsigma(U_0) = \s(\widehat U_0)$. We write $V_0 = \varsigma(U_0)$, and {\small NONEMPTY} plays $V_0$ to complete round $0$ of $\mathrm{BM}(X)$. 

To begin round $1$, {\small EMPTY} plays some $U_1 \sub V_0$. Similarly to round $0$, let $\widehat U_1 = \psi_{\pi(U_1)}(\{\pi(U_0),\pi(U_1)\})$, and define $\varsigma(U_0,U_1) = \sigma(\widehat U_0,V_0,\widehat U_1)$. (Note that $\widehat U_0$, $V_0$, and $\widehat U_1$ are all functions of $U_0$ and $U_1$.) We write $V_1 = \varsigma(U_0,U_1)$, and {\small NONEMPTY} plays $V_1$ to complete round $1$ of $\mathrm{BM}(X)$.

To begin round $2$, {\small EMPTY} plays some $U_2 \sub V_1$, and then {\small NONEMPTY} must respond based only on {\small EMPTY}'s two previous moves, $U_1$ and $U_2$. Observe that $U_2 \sub V_1 \sub \widehat U_1$; hence $U_2$ is contained in exactly one member of the cellular family $\mathrm{range}(\psi_{\pi(U_1)})$, namely $\widehat U_1$. But $\psi_{\pi(U_1)}$ is injective, and $\psi_{\pi(U_1)}^{-1}(\widehat U_1) = \{\pi(U_0),\pi(U_1)\}$. Thus {\small NONEMPTY} is able to recover $\pi(U_0)$ by knowing $U_1$ and $U_2$. With this information {\small NONEMPTY} is able to reconstruct $\widehat U_0$ and $V_0$ as well (by simulating the gameplay described in the previous two paragraphs). The rest proceeds just as in round $1$: let $\widehat U_2 = \psi_{\pi(U_2)}(\{\pi(U_0),\pi(U_1),\pi(U_2)\})$, and define $\varsigma(U_1,U_2) = \sigma(\widehat U_0,V_0,\widehat U_1,V_1,\widehat U_2)$. We write $V_2 = \varsigma(U_1,U_2)$, and {\small NONEMPTY} plays $V_2$ to complete round $2$.

All later rounds are played in a similar fashion. {\small EMPTY} plays some $U_n \sub V_{n-1}$ to begin round $n$, and then {\small NONEMPTY} must respond based only on {\small EMPTY}'s two previous moves, $U_{n-1}$ and $U_n$. But as before, $U_n$ is contained in exactly one member of the cellular family $\mathrm{range}(\psi_{\pi(U_{n-1})})$, namely $\hat U_{n-1}$. As $\psi_{\pi(U_{n-1})}^{-1}(\widehat U_{n-1}) = \{\pi(U_0),\pi(U_1),\dots,\pi(U_{n-1})\}$, {\small NONEMPTY} is able to recover the sets $\pi(U_0),\pi(U_1),\dots,\pi(U_{n-2})$ by knowing only $U_{n-1}$ and $U_n$. With this information {\small NONEMPTY} is able to reconstruct $\widehat U_i$ and $V_i$ for all $i < n$. Then, we let $\widehat U_n = \psi_{\pi(U_n)}(\{\pi(U_0),\pi(U_1),\dots,\pi(U_n)\})$, and define $\varsigma(U_{n-1},U_n) = \s(\widehat U_0,V_0,\dots,\widehat U_{n-1},V_{n-1},\widehat U_n)$. We write $V_n = \varsigma(U_{n-1},U_n)$, and {\small NONEMPTY} plays $V_n$ to complete round $n$ of $\mathrm{BM}(X)$.

To see that $\varsigma$ is a winning strategy for {\small NONEMPTY}, we must show that $\bigcap_{n \in \w} U_n = \bigcap_{n \in \w} V_n = \bigcap_{n \in \w}\widehat U_n \neq \0$. To see this, consider the play of $\mathrm{BM}(X)$ where {\small EMPTY} plays $\widehat U_n$ in round $n$, and {\small NONEMPTY} responds by playing $V_n$. This is clearly a valid play of $\mathrm{BM}(X)$, and as $V_n = \s(\widehat U_0,V_0,\dots,\widehat U_{n-1},V_{n-1},\widehat U_n)$ for all $n$, {\small NONEMPTY} plays this game according to the winning strategy $\s$. Thus {\small NONEMPTY} wins this play of $\mathrm{BM}(X)$, and this means $\bigcap_{n \in \w} V_n = \bigcap_{n \in \w}\widehat U_n \neq \0$.
\end{proof}

Let us point out that the hypothesis of Theorem~\ref{thm:galvin} can be weakened slightly: our proof shows that it is enough that the collection of all nonempty open $U \sub X$ with $\pnt(U) \leq S(U)$ forms a $\pi$-base for $X$.

\vspace{3mm}

\begin{definition}
Let $\PP$ be a partially ordered set.
We write $q \leq p$ to mean that $q$ extends $p$.
$\PP$ is called \emph{separative} if for all $p,q \in \PP$, if $q \not\leq p$ then there is some $r \leq q$ such that $r$ and $p$ are incompatible (denoted $r \perp p$).

A poset $\PP$ has the $\k$-cc if every antichain in $\PP$ has size less than $\k$. 
The \emph{Souslin number} of $\PP$, denoted $S(\PP)$, is defined as
$$S(\PP) = \min \set{\k}{\PP \text{ has the }\k\text{-cc}.}$$
A subset $\DD$ of $\PP$ is \emph{dense} if for every $p \in \PP$, there is some $q \in \DD$ with $q \leq p$.
The \emph{Noetherian type} of any $\DD \sub \PP$, denoted $\nt(\DD)$, is 
$$\nt(\DD) = \min \set{\k}{\text{for all } q \in \PP, \card{\set{p \in \DD}{q \leq p}} < \k}.$$
The $\pi$-Noetherian type of $\PP$ is
$$\pnt(\PP) = \min \set{\nt(\DD)}{\DD \text{ is dense in }\PP}.$$
\end{definition}

\vspace{2mm}

Note that the $\pi$-Noetherian type and Souslin number of a regular space $X$, as defined above, are the $\pi$-Noetherian type and Souslin number, respectively, of the partial order of open subsets of $X$, ordered by inclusion.

\begin{definition}
For any partially ordered set $\PP$, let $\axiom(\PP)$ denote the statement
$\pnt(\PP) \leq S(\PP)$.
If $\mathcal K$ is a class of partial orders, then $\axiom(\mathcal K)$ denotes the statement that $\axiom(\PP)$ holds for all $\PP \in \mathcal K$. The symbol \axiom abbreviates the statement $\axiom(\mathrm{separative})$.
\end{definition}

\begin{remark}
  The statement $\axiom(\text{all posets})$ is false. For example, any ordinal $\a$, ordered by $\geq$, has $S(\a) = 2$ and $\pnt(\a) = \cf(\a)$. To get examples with larger Souslin number, consider a union of incompatible chains:
  $\k \times \lambda$ with $(\a,\b)\leq(\a',\b')$ if and only if $\a=\a'$ and $\b\geq\b'$.
  With respect to this ordering, $S(\k \times \lambda) = \k^+$ and $\pnt(\k \times \lambda) = \cf(\lambda)$.
\end{remark}

\begin{remark}\label{rem:countable=dumb}
The statement $\axiom(\text{separative}+\text{countable})$ is true. If $\PP$ is a finite separative poset, then every $p \in \PP$ has a minimal extension, and setting $\DD$ equal to the set of all minimal elements of $\PP$ shows $\pnt(\PP) = 2$. As every (nonempty) poset has Souslin number $\geq\! 2$, this shows \axiom$\!\!(\PP)$ holds. If $\PP$ is countably infinite, then separativity implies $S(\PP) = \aleph_1$. But clearly $\pnt(\PP) \leq \card{\PP}^+$ for any poset $\PP$ (by setting $\DD = \PP$), so again \axiom$\!\!(\PP)$ holds.
\end{remark}

\begin{remark}\label{rem:D}
The statement $\axiom(\text{separative}+\text{cardinality}\leq\!\aleph_1)$ is true. If $\card{\PP} = \aleph_1$, write $\PP = \set{p_\a}{\a < \w_1}$ and let 
$\DD = \set{p_\a}{\text{if } \xi < \a \text{ then } p_\xi \not\leq p_\a}.$ 
If $p = p_\a \in \PP$, then letting $\xi \leq \a$ be the least ordinal such that $p_\xi \leq p_\a$, we have $p_\xi \in \DD$. Thus $\DD$ contains an extension of $p$, and as $p$ was arbitrary, this shows $\DD$ is dense in $\PP$. But for any $p_\a \in \PP$, clearly $\set{q \in \DD}{p_\a \leq q} \sub \set{p_\xi}{\xi < \a}$, and so $\nt(\DD) \leq \aleph_1$. Hence $\pnt(\PP) \leq \aleph_1$. As in the previous remark, every infinite separative poset has uncountable Souslin number; it follows that $\pnt(\PP) \leq S(\PP)$.
\end{remark}

These remarks show that $\aleph_2$ is the least cardinality of a separative poset $\PP$ for which \axiom$\!\!(\PP)$ can fail. We will see in Section~\ref{sec:independence} that such a failure is consistent.

The next theorem shows that \axiom is just a translation of the hypotheses of Theorem~\ref{thm:galvin} into the language of posets. 
The ideas used to prove this theorem are entirely standard. There is a large literature concerning infinite games (including the Banach-Mazur game) on partial orders and Boolean algebras \cite{Jech, Foreman, Velickovic}. It is implicit in this literature that the (topological) Banach-Mazur game on a space $X$ is essentially equivalent to the (order-theoretic) Banach-Mazur game on any $\pi$-base $\B$ for $X$, ordered by inclusion. 
The proof of Theorem~\ref{thm:translation} just expresses one aspect of this equivalence.

A topological space $X$ is called \emph{quasi-regular} if it is Hausdorff and, for every nonempty open $U \sub X$, there is a nonempty open $V$ with $\closure{V} \sub U$. (In some places such spaces are called \emph{$\pi$-regular}.) Note that every $T_3$ space is quasi-regular, but there are quasi-regular spaces that are not $T_3$. (For an example, consider the topology on $\R$ generated by the usual topology plus the set $\R \setminus \set{\frac{1}{n}}{n \in \N}$.)

\begin{theorem}\label{thm:translation}
The statement \axiom is equivalent to the statements:
\begin{enumerate}
\item For every Boolean algebra $\BB$, $\pnt(\BB) \leq S(\BB)$.
\item For every Stone space (i.e., every zero-dimensional compact Hausdorff space) $X$, $\pnt(X) \leq S(X)$.
\item For every quasi-regular space $X$, $\pnt(X) \leq S(X)$.
\end{enumerate}
\end{theorem}
\begin{proof}
To see that $(1)$ implies \axiom, let $\PP$ be a separative partial order. Every separative partial order embeds densely into a (unique) complete Boolean algebra, known as its completion. Let $\BB$ denote the completion of $\PP$. By $(1)$, $\pnt(\BB) \leq S(\BB)$. It is straightforward to show that $S(\PP) = S(\BB)$ and $\pnt(\PP) = \pnt(\BB)$, so $\pnt(\PP) \leq S(\PP)$. As $\PP$ was arbitrary, \axiom follows.

To see that $(2)$ implies $(1)$, let $\BB$ be a Boolean algebra and let $X$ denote its Stone space. By $(2)$, $\pnt(X) \leq S(X)$. It is straightforward to show that $S(\BB) = S(X)$ and $\pnt(\BB) = \pnt(X)$, so $\pnt(\BB) \leq S(\BB)$ and $(1)$ follows.

Clearly $(3)$ implies $(2)$, because every compact Hausdorff space is $T_4$ and hence quasi-regular.

To see that \axiom implies $(3)$, let $X$ be a quasi-regular space. Recall that a subset $U$ of $X$ is \emph{regular open} if $U = \mathrm{int}(\closure{U})$. Let $\PP = \ro(X) \setminus \{\0\}$, ordered by inclusion. It is straightforward to check that the quasi-regularity of $X$ implies $\PP$ is separative. Thus, by \axiom, $\pnt(\PP) \leq S(\PP)$. Using the quasi-regularity of $X$ again, it is straightforward to check that $\pnt(\PP) = \pnt(X)$ and $S(\PP) = S(X)$. As $X$ was arbitrary, $(3)$ follows.
%
\end{proof}

\begin{corollary}\label{cor:teeitup}
Suppose \axiom holds. Then for every quasi-regular space $X$, \emph{\small NONEMPTY} has a winning strategy in $\mathrm{BM}(X)$ if and only if \emph{\small NONEMPTY} has a winning $2$-tactic.
\end{corollary}
\begin{proof}
This follows immediately from Theorems \ref{thm:galvin} and \ref{thm:translation}.
\end{proof}

The ``quasi-regular'' hypothesis in Theorem~\ref{thm:translation} cannot be omitted. To see this, note that \axiom will be shown to be consistent in the next section, and yet there are unconditional \zfc examples of spaces $X$ (necessarily not quasi-regular) with $S(X) < \pnt(X)$. An easy $T_1$ is example is the co-finite topology on an infinite set $X$, where $S(X) = 2$ and $\pnt(X) = |X|$. The following example shows a $T_2$ (in fact, Urysohn and completely Hausdorff) space $X$ with $S(X) < \pnt(X)$. The example is a modification of a space described by Debs \cite[p. 235]{Telgarsky}; a related example with the same properties was communicated to the fourth author by Bill Fleissner. 

\begin{example}
Let $\k$ be an uncountable regular cardinal, and let $X = 2^\k$. Let $\s$ denote the usual product topology on $X$. Define a new topology $\t$ on $X$ by declaring $V \in \t$ if and only if $V = U \setminus A$ for some $U \in \s$ and some $A \sub X$ with $|A| \leq \k$. This topology $\t$ is Hausdorff, and in fact Urysohn and completely Hausdorff, because $\s$ has these properties and all these properties are preserved by refinement. Note that any cellular family $\C$ in $\t$ gives rise to a cellular family of the same size in $\s$, namely $\set{\mathrm{int}_{\s}(\closure{V}^\s)}{V \in \C}$. Because $\s$ has the ccc, this means $\t$ also has the ccc; or in other words, $S(X_\t) = \aleph_1$. Now, we claim $\pnt(X_\t) > \k$. To see this, let $\B$ be a $\pi$-base for $(X,\t)$. Observe that if $\set{U_\a \setminus A_\a}{\a < \k}$ is any $\k$-sized subset of $\t$, then it is not a $\pi$-base, because no member of this collection is a subset of the nonempty $\t$-open set $X \setminus B$, where $B$ is any $\k$-sized, dense (with respect to $\s$) subset of $X$ disjoint from $\bigcup_{\a < \k}A_\a$. Thus $|\B| > \k$. For each $V \in \B$, fix $U_V \in \s$ such that $V = U_V \setminus A$ for some $A \sub X$ with $|A| \leq \k$. Let $\A$ denote the standard basis of clopen subsets of $(X,\s)$. As $|\A| = \k$, there is by the pigeonhole principle some $U \in \A$ such that $\card{\set{V \in \B}{U \sub U_V}} > \k$. Enumerate some $\k$-sized subset $\set{U_\a \setminus A_\a}{\a < \k}$ of $\B$ such that $U \sub U_\a$ for all $\a$, and let $W = U \setminus \bigcup_{\a < \k}A_\a$. Clearly $\0 \neq W \in \t$, and $\card{\set{V \in \B}{W \sub V}} \geq \k$. Hence $\nt(\B) > \k$. As $\B$ was an arbitrary $\pi$-base for $(X,\t)$, this shows $\pnt(X_\t) > \k$.
\end{example}

\section{Higher Davies trees, and the consistency of \axiom}\label{sec:V=L}

A Davies tree is a sequence $\seq{M_\a}{\a < \mu}$ of elementary submodels of some large fragment $H_\theta$ of the set-theoretic universe such that the $M_\a$ enjoy certain coherence and covering properties. (These sequences are called ``trees'' because they are usually constructed by enumerating the leaves of a tree of elementary submodels of $H_\theta$.) These structures provide a unified framework for carrying out a wide variety of constructions in infinite combinatorics. They were introduced by R. O. Davies in \cite{Davies}, and an excellent survey of their many uses can be found in Daniel and Lajos Soukup's paper \cite{Soukup^2}.

Also in \cite{Soukup^2}, the Soukups construct a countably closed version of a Davies tree called a ``high Davies tree.'' Their construction uses $\gch +\square$ and is rather intricate. In this section, we exposit a simpler construction that also produces high Davies trees (and in fact, the stronger version called ``sage'' Davies trees), this time using \gch and a parametrized version of the Very Weak Square principle introduced by Foreman and Magidor in \cite{Foreman&Magidor}. We note that while our construction is simpler, the proof that the construction actually works is fairly involved, so that we are not really getting high/sage Davies trees for less work overall. 
Rather, the advantage of our construction is that it generalizes readily to higher cardinals, so that we are able to obtain $\k$-closed versions of high Davies trees for uncountable $\k$. We call these structures ``higher'' Davies trees. Our primary motivation for constructing these higher Davies trees is that, while the existence of high Davies trees enable us to prove \axiom$\!\!(\text{ccc}+\text{separative})$, the higher versions seem to be required for handling posets with larger Souslin number.

We begin this section by defining our generalization of high Davies trees. We then show in Theorem~\ref{thm:kcc} that the existence of these higher Davies trees implies \axiom, and thus, via Corollary~\ref{cor:teeitup}, the failure of Telg\'arsky's conjecture.
After this, we show how to construct the higher Davies trees using \gch plus a weakening of $\square$. 

In what follows, $H_\theta$ denotes the set of all sets hereditarily smaller than some very big cardinal $\theta$. 
Given two sets $M$ and $N$, we write $M \prec N$ to mean that $(M,\in)$ is an elementary submodel of $(N,\in)$.

A set $M$ is called \emph{$<\!\k$-closed} if $M^{<\k} \sub M$. If $M$ satisfies (enough of) \zfc, this is equivalent to the property $[M]^{<\k} \sub M$. The following two facts will be used in what follows: If $M \prec H_\theta$ and $M$ is $<\!\k$-closed, then
\begin{itemize}
\item[$\circ$] $\k \sub M$, and
\item[$\circ$] if $p \in M$ and $\card{p} \geq \k$, then $\card{p \cap M} \geq \k$.
\end{itemize}

\begin{definition}\label{def:tree}
Let $\k,\mu$ be infinite cardinals and let $p$ be some set. A \emph{$\k$-high Davies tree for $\mu$ over $p$} is a sequence $\seq{M_\a}{\a < \mu}$ of elementary submodels of $(H_\theta,\in)$, for some ``big enough'' regular cardinal $\theta$, such that
\begin{enumerate}
\item $p \in M_\a$, $M_\a$ is $<\!\k$-closed, and $\card{M_\a} = \k$ for all $\a < \mu$.
\item $[\mu]^{<\k} \sub \bigcup_{\a < \mu}M_\a$.
\item For each $\a < \mu$, there is a set $\mathcal N_\a$ of elementary submodels of $H_\theta$ such that $\card{\mathcal N_\a} < \k$, each $N \in \mathcal N_\a$ is $<\!\k$-closed and contains $p$, and
$$\textstyle \bigcup_{\xi < \a}M_\xi \,=\, \bigcup \mathcal N_\a.$$
\end{enumerate}
\end{definition}

Setting $\k = \aleph_0$, an $\aleph_0$-high Davies tree for $\mu$ over $p$ is just called a \emph{Davies tree for $\mu$ over $p$}. The existence of these objects is a theorem of \zfc for any value of $\mu$ and any parameter $p$ \cite[Theorem 3.1]{Soukup^2}.
Assuming \ch and setting $\k = \aleph_1$, an $\aleph_1$-high Davies tree for $\mu$ over $p$ is called a \emph{high Davies tree for $\mu$ over $p$}. The existence of these trees is independent of \zfc: their existence is guaranteed by $\gch+\square$ for any values of $\mu$ and $p$ \cite[Theorem 8.1]{Soukup^2}, but it is consistent (relative to a supercompact cardinal) that there are no high Davies trees for any $\mu > \aleph_\w$ \cite[Corollary 9.2]{Soukup^2}.

We postpone the proof that the existence of $\k$-high Davies trees is consistent until later in this section, and turn now to the relatively short proof that their existence implies \axiom.

Given a poset $\PP$ and $Q \sub \PP$, write $Q \down \, = \set{p \in \PP}{p \leq q \text{ for all } q \in Q}$. 

\begin{lemma}\label{lem:kcc}
Let $\PP$ be a separative partial order with the $\k$-cc, and let $Q \sub \PP$. There is some $Q' \sub Q$ with $\card{Q'} < \k$ such that $Q' \down \, = Q \down$.
\end{lemma}
\begin{proof}
Aiming for a contradiction, suppose $Q \sub \PP$ has cardinality at least $\k$, and that there is no $Q' \sub Q$ with $\card{Q'} < \k$ such that $Q' \down\, = Q \down$. We may then find a sequence $\seq{q_\a}{\a < \k}$ of members of $Q$ such that 
$$\set{q_\xi}{\xi < \a} \down\, \neq \set{q_\xi}{\xi < \b} \down \ \text{ whenever } \ \a < \b < \k.$$ 
For each $\a < \k$, fix some 
$p_\a \in (\set{q_\xi}{\xi < \a} \down) \setminus (\set{q_\xi}{\xi < \a+1} \down).$
Then $p_\a \leq q_\xi$ for all $\xi < \a$, but $p_\a \not\leq q_\a$. By separativity, there is some $r_\a \leq p_\a$ such that $r_\a \perp q_\a$. But then
$\set{r_\a}{\a < \k}$
is an antichain in $\PP$, because if $\a < \b$ then $r_\a \perp q_\a$ while $r_\b \leq q_\a$, which implies $r_\a \perp r_\b$.
\end{proof}

\begin{theorem}\label{thm:kcc}
Let $\PP$ be a separative poset with $S(\PP) = \k$. If there is a $\k$-high Davies tree for some $\mu \geq |\PP|$, then \axiom$\!\!(\PP)$ holds. 
\end{theorem}

\begin{proof}
%
Fix a cardinal $\mu \geq |\PP|$, and let $\seq{M_\a}{\a < \mu}$ be a $\k$-high Davies tree for $\mu$ over $\PP$. 
Formally, we may take the members of $\PP$ to be ordinals $< \!\mu$; this ensures, via property $(2)$ in Definition~\ref{def:tree}, that $\PP \sub \bigcup_{\a < \mu}M_\a$. 
(Informally, we avoid this identification: the letters $\a,\z,\xi$ will be reserved for ordinals as such, and not used for members of $\PP$.)

Fix a well-ordering $\sqsubset$ of $\PP$ such that for each $\a < \mu$,
\begin{itemize}
\item[$\circ$] If $p \in \PP \cap M_\a$ and $q \notin \bigcup_{\xi \leq \a}M_\xi$, then $p \sqsubset q$.
\item[$\circ$] the restriction of $\sqsubset$ to $M_\a$ is a well ordering of $\PP \cap (M_\a \setminus \bigcup_{\xi < \a}M_\xi)$ with order type $\leq \! \k$.
\end{itemize}
It is easy to construct such a well-ordering from our $\k$-high Davies tree, using the fact that $|M_\a| = \k$ for each $\a$.
Define
$$\DD = \set{p \in \PP}{\text{if } p \sqsubset q \text{ then } p \not\leq q}.$$
Note the similarity of this definition with the one in Remark~\ref{rem:D}.
We claim that $\DD$ is a dense subset of $\PP$ and that $\nt(\DD) \leq \k$.

Given $q \in \PP$, let $p$ be the $\sqsubset$-least member of $\{q\} \down \,= \set{p \in \PP}{p \leq q}$. Then $p \in \DD$, so $\DD$ contains an extension of $q$. As $q$ was arbitrary, $\DD$ is dense in $\PP$.

To show $\nt(\DD) \leq \k$, let us aim for a contradiction and suppose not. This means $\set{q \in \DD}{p \leq q} \geq \k$ for some $p \in \PP$.
Henceforth, let $p$ denote the $\sqsubset$-minimal member of $\PP$ with this property, and let $Q = \set{q \in \DD}{p \leq q}$. Fix $\a < \mu$ with $p \in M_\a \setminus \bigcup_{\xi < \a}M_\xi$.

By our definitions of $\DD$ and of $\sqsubset$,
$$\textstyle Q \ \sub \ \set{q \in \PP}{q \sqsubseteq p} \ = \ \bigcup \set{\PP \cap M_\xi}{\xi < \a} \cup \set{q \in \PP \cap M_\a}{q \sqsubset p}.$$
But $\card{\set{q \in \PP \cap M_\a}{q \sqsubset p}} < \k$ (by our definition of $\sqsubset$) and $\card{Q} \geq \k$, and it follows that $\card{Q \cap \bigcup \set{M_\xi}{\xi < \a}} \geq \k$. Fix a set $\mathcal N$ as described in property $(3)$ of Definition~\ref{def:tree}. 

Because $\card{Q \cap \bigcup \set{M_\a}{\a < \b}} \geq \k$ and $\card{\mathcal N} < \k$, we have $\card{Q \cap N} \geq \k$ for some $N \in \mathcal N$.
By Lemma~\ref{lem:kcc}, there is some $Q' \sub Q \cap N$ with $\card{Q'} < \k$ such that $Q' \down \, = (Q \cap N) \down$. Because $N$ is $<\!\k$-closed, we have $Q' \in N$. 

Because $p \leq q$ for all $q \in Q'$, the statement ``$Q' \down \, \neq \0$'' is true in $H_\theta$.
By elementarity, $N \models$ ``$Q' \down \, \neq \0$'', and thus there is some $\bar p \in N$ such that $\bar p \in Q' \down \, = (Q \cap N) \down$. Our definition of $\DD$ implies $q \sqsubseteq \bar p$ for every $q \in Q \cap N$, because $\bar p \leq q$ and $q \in \DD$. In particular, $\set{q \in \DD}{\bar p \leq q} \supseteq Q \cap N$, and so $\card{\set{q \in \DD}{\bar p \leq q}} \geq \k$. But $\bar p \in \bigcup \mathcal N = \bigcup_{\xi < \a}M_\xi$ and thus $\bar p \sqsubset p$, contradicting our choice of $p$. Therefore $\nt(\DD) \leq \k$, and this implies that $\pnt(\PP) \leq \k = S(\PP)$.
\end{proof}

\begin{corollary}\label{cor:trees}
Suppose that for any set $p$ and any regular cardinal $\k$, there are arbitrarily high values of $\mu$ for which there is a $\k$-high Davies tree for $\mu$ over $p$. Then \axiom holds.
\end{corollary}
\begin{proof}
Erd\H{o}s and Tarski proved in \cite{Erdos&Tarski} that the Souslin number of any poset is a regular cardinal. Using this fact, the corollary follows directly from Theorem~\ref{thm:kcc}. 
\end{proof}

Let us point out that we have proved something a little stronger than claimed. The $\k$-cc and separativity were used only to prove Lemma~\ref{lem:kcc}, and were not mentioned otherwise in the proof of Theorem~\ref{thm:kcc}. So we have really proved that the hypotheses of Corollary~\ref{cor:trees} imply that \axiom holds for all posets satisfying the conclusion of Lemma~\ref{lem:kcc}. This class of posets is strictly broader than the class of $\k$-cc posets, even if we restrict our attention to separative posets. For example, the poset $\PP$ of all infinite closed subsets of the Baire space (ordered by inclusion) is far from ccc, but it still has the property that for any $Q \sub \PP$, there is a countable $Q' \sub Q$ with $Q' \down \,= Q \down$.

\subsection*{The construction of the higher Davies trees}

We turn now to the proof that the hypothesis of Corollary~\ref{cor:trees} is consistent.
In fact, we will prove a little more by constructing the following stronger version of $\k$-high Davies trees.

\begin{definition}\label{def:sage}
A \emph{$\k$-sage Davies tree for $\mu$ over $p$} is a $\k$-high Davies tree for $\mu$ over $p$ (cf. Definition~\ref{def:tree}) satisfying the following two additional properties:
\begin{enumerate}
\setcounter{enumi}{3}
\item $\seq{M_\a}{\a < \b} \in M_\b$ for each $\b < \mu$.
\item $\bigcup_{\a < \mu} M_\a$ is a $<\!\k$-closed elementary submodel of $H_\theta$.
\end{enumerate}
\end{definition}

Assuming \ch and setting $\k = \aleph_1$, an $\aleph_1$-sage Davies tree for $\mu$ over $p$ is called a \emph{sage Davies tree for $\mu$ over $p$}.
High Davies trees suffice for many interesting applications, while some applications seem to require the stronger sage Davies trees: see, e.g., \cite[Section 13]{Soukup^2}. 
Theorem~\ref{thm:kcc} and Corollary~\ref{cor:trees} above show that $\k$-high trees are sufficient for our purposes here. We consider $\k$-sage Davies trees nonetheless because properties $(4)$ and $(5)$ come at no extra cost in our construction below, and they may be useful in future applications of these new structures.

A set $M$ is called \emph{weakly $<\!\lambda$-closed} if $M \cap H_\theta^{<\lambda} \sub M^{<\lambda}$. If $M$ satisfies (enough of) \zfc, this is equivalent to the property $M \cap [H_\theta]^{<\lambda} \sub [M]^{<\lambda}$, i.e., $<\!\lambda$-sized elements of $M$ are also subsets of $M$. 

\begin{definition}\label{def:approx}
Let $\lambda$ be a regular uncountable cardinal and $p$ any set. A \emph{long $\lambda$-approximation sequence over $p$} is a sequence $\seq{M_\a}{\a < \mu}$ of elementary submodels of $H_\theta$, for some ordinal $\mu$, such that
\begin{enumerate}
\item For each $\a < \mu$, $p \in M_\a$, $M_\a$ is weakly $<\!\lambda$-closed, and $\card{M_\a} < \lambda$.
\item For each $\a < \mu$, $\seq{M_\xi}{\xi < \a} \in M_\a$.
\end{enumerate}
Furthermore, if $\lambda = \k^+$, then $\seq{M_\a}{\a < \mu}$ is said to be \emph{closed} if
\begin{enumerate}
\item[$(3)$] For each $\a < \mu$, $M_\a$ is $<\!\k$-closed.
\end{enumerate}
Note that we place no restrictions on $\mu$. In particular, the empty sequence is considered a long $\lambda$-approximation sequence (over any $p$), by setting $\mu=0$.
\end{definition}

Long $\k^+$-approximation sequences can be thought of as a very weak form of $\k$-sage Davies trees. Examining the definitions, we see that if a sequence $\seq{M_\a}{\a < \mu}$ is a closed long $\k^+$-approximation sequence over $p$, then it satisfies every part of the definition of a $\k$-sage Davies tree for $\mu$ over $p$ except perhaps for conditions $(2)$, $(3)$, and $(5)$. 
Our strategy for getting $\k$-sage Davies trees is to show that \gch plus (a weakening of) $\square$ implies every $\mu$-length closed long $\k^+$-approximation sequence is already a $\k$-sage Davies tree for $\mu$. This strategy generalizes a result proved by the third author in \cite[Lemma 3.17]{Dave1}, which is essentially the case $\k = \aleph_0$:

\begin{theorem*}[Milovich, \cite{Dave1} 2008]
Every long $\aleph_1$-approximation sequence is a Davies tree.
\end{theorem*}

The following lemma collects some easy-to-prove facts concerning these approximation sequences.

\begin{lemma}\label{lem:facts} 
Let $\lambda$ be a regular uncountable cardinal.
\begin{enumerate}
\item If $\seq{M_\a}{\a < \mu}$ is a long $\lambda$-approximation sequence, so is $\seq{M_\xi}{\xi < \a}$ for any $\a < \mu$.
\item If $\seq{M_\a}{\a < \mu}$ is a long $\lambda$-approximation sequence, then $\a \in M_\a$ for all $\a < \mu$.
\item If $\seq{M_\a}{\a < \mu}$ is a long $\lambda$-approximation sequence and $\a < \b < \mu$, 
$$\a \in M_\b \quad \Leftrightarrow \quad M_\a \in M_\b \quad \Leftrightarrow \quad M_\a \sub M_\b.$$
\item For every ordinal $\mu$ and every set $p$, there is a long $\lambda$-approximation sequence over $p$ of length $\mu$.
\item Assume \gch, and suppose $\lambda = \k^+$ for some regular cardinal $\k$. For every ordinal $\mu$ and every set $p$, there is a closed long $\k^+$-approximation sequence over $p$ of length $\mu$.
\end{enumerate}
\end{lemma}
\begin{proof}
$(1)$ is evident from the definitions.

For $(2)$, note that $\a$ is definable from $\seq{M_\xi}{\xi < \a}$. As $\seq{M_\xi}{\xi < \a} \in M_\a$ and $M_\a \prec H_\theta$, this means $\a \in M_\a$.

For $(3)$, let $\a < \b < \mu$. Observe that $\a \in M_\b$ and $\seq{M_\xi}{\xi < \b} \in M_\b$ implies $M_\a \in M_\b$. Then, $M_\a \in M_\b$ implies $M_\a \sub M_\b$ because $\card{M_\a} < \lambda$ and $M_\b$ is weakly $\lambda$-closed. Finally, $(2)$ implies that if $M_\a \sub M_\b$ then $\a \in M_\b$.

For $(4)$ (and $(5)$), we use transfinite recursion with a closing-off argument at each stage. Let $M_0$ be any weakly $\lambda$-closed elementary submodel of $H_\theta$ containing $p$, such that $\card{M_0} < \lambda$ (and for $(5)$, let us also insist that $M_0$ is $<\!\k$-closed). The existence of such a model uses a standard closing-off argument (plus the equality $\k = \k^{<\k}$, which follows from \gch, if we insist that $M_0$ is $<\!\k$-closed). At stage $\a > 0$ of the recursion, let us assume that $\seq{M_\xi}{\xi < \a}$ has already been constructed. Using the standard closing-off argument again, there is a weakly $\lambda$-closed $M_\a \prec H_\theta$ such that $p, \seq{M_\xi}{\xi < \a} \in M_\a$ and $|M_\a| < \lambda$. (If \gch holds, we may also insist that $M_\a$ is $<\!\k$-closed.) The sequence $\seq{M_\a}{\a < \mu}$ produced by this recursion has all the desired properties.
\end{proof}

Thus we see that (closed) long $\k^+$-approximation sequences of any length are easy to construct using $\zfc(+\gch)$ for any regular infinite cardinal $\k$. On the other hand, basic cardinal arithmetic shows that there are no nonempty closed long $\k^+$-approximation sequences if $\k^{<\k} > \k$.

\begin{definition}\label{def:daviesproperty}
Let $\eta, \k$ be infinite regular cardinals with $\k \leq \eta$.
\begin{itemize}
\item[$\circ$] A set $S$ is \emph{$\k$-directed} if for every $T \in [S]^{<\k}$, there is some $x \in S$ such that $\bigcup T \sub x$.
\item[$\circ$] A sequence $\seq{M_\a}{\a < \mu}$ has the \emph{$(\eta,\k)$-Davies property} if for each $\a \leq \mu$, there is a set $\mathcal N_\a$ such that
\begin{itemize}
\item each $N \in \mathcal N_\a$ is the union of a $\k$-directed subset of $\set{M_\xi}{\xi < \a}$.
\item $\card{\mathcal N_\a} < \eta$.
\item $\bigcup_{\xi < \a}M_\xi = \bigcup \mathcal N_\a.$
\end{itemize}
\item[$\circ$] A sequence $\seq{M_\a}{\a < \mu}$ has the \emph{$\k$-Davies property} if it has the $(\k,\k)$-Davies property.
\end{itemize}
\end{definition}

The following two easy-to-prove facts concerning directed unions of elementary submodels will be used in what follows:
\begin{itemize}
\item[$\circ$] If $S$ is a $\k$-directed set for some infinite $\k$, and if $M \prec H_\theta$ for each $M \in S$, then $\bigcup S \prec H_\theta$.
\item[$\circ$] If furthermore each $M \in S$ is $<\!\k$-closed, then $\bigcup S$ is $<\!\k$-closed.
\end{itemize}

\begin{lemma}\label{lem:sage}
Suppose that $\k, \mu$ are infinite cardinals with $\cf(\mu) \geq \k$, and that $\seq{M_\a}{\a < \mu}$ is a closed long $\k^+$-approximation sequence over some $p$ with the $\k$-Davies property. If $\set{M_\a}{\a < \mu}$ is $\k$-directed, then $\seq{M_\a}{\a < \mu}$ is a $\k$-high Davies tree for $\mu$ over $p$.
\end{lemma}
\begin{proof}
As we mentioned already, it suffices to check conditions $(2)$, $(3)$, and $(5)$ from Definitions \ref{def:tree} and \ref{def:sage}.

To prove $(2)$, let $X \sub \mu$ with $\card{X} < \k$. Because $\set{M_\a}{\a < \mu}$ is $\k$-directed, there is some $\a < \mu$ such that $\bigcup \set{M_\xi}{\xi \in X} \sub M_\a$. 
But $M_\xi \sub M_\a$ implies $\xi \in M_\a$, so $X \sub M_\a$. Because $M_\a$ is $<\!\k$-closed, $X \in M_\a$.

To prove $(3)$, fix $\a < \mu$ and let $\mathcal N_\a$ be as in the definition of the $\k$-Davies property. Each $N \in \mathcal N_\a$ is the union of a $\k$-directed set of $<\!\k$-closed elementary submodels of $H_\theta$, each containing $p$, which implies $N$ is a $<\!\k$-closed elementary submodel of $H_\theta$ and $p \in N$. The definition of the $\k$-Davies property also gives $\bigcup \mathcal N_\a = \bigcup_{\xi < \a}M_\xi$, so $(3)$ is satisfied.

To prove $(5)$, simply note that $\bigcup \set{M_\a}{\a < \mu}$ is a $\k$-directed union of $<\!\k$-closed elementary submodels of $H_\theta$.
\end{proof}

Our goal now is to get sequences satisfying the hypotheses of Lemma~\ref{lem:sage}, or rather to show that $\gch+\square$ implies every closed long $\k^+$-approximation sequence already satisfies these hypotheses. Our proof exploits the cardinal normal form, a variant of the Cantor normal form introduced in \cite{Dave2}. 

\begin{definition}
$\ $
\begin{itemize}
\item[$\circ$] An ordinal $\a$ is \emph{cardinally even} if $\a = \card{\a} \cdot \b$ for some ordinal $\b$, where $\cdot$ denotes ordinal multiplication. 
\item[$\circ$] A finite (possibly empty) formal sum $\dlt_0+\dots+\dlt_{n-1}$ is a \emph{cardinal normal form} if each $\dlt_i$ is cardinally even and $\card{\dlt_0} > \dots > \card{\dlt_{n-1}} > 0$. If $\a = \dlt_0+\dots+\dlt_{n-1}$, we say that $\dlt_0+\dots+\dlt_{n-1}$ is the cardinal normal form of $\a$.
\item[$\circ$] Given a cardinal $\lambda$, a finite (possibly empty) sum $\dlt_0+\dots+\dlt_{n-1}+\dlt_n$ is called a \emph{$\lambda$-truncated cardinal normal form} if either it is a cardinal normal form with ${\dlt_n} \geq \lambda$, or else the (possibly empty) sub-sum $\dlt_0+\dots+\dlt_{n-1}$ is a cardinal normal form with $\dlt_{n-1} \geq \lambda$, and $\dlt_n$ is an ordinal less than $\lambda$. 
If $\a = \dlt_0+\dots+\dlt_{n-1}+\dlt_n$, we say that $\dlt_0+\dots+\dlt_n$ is the $\lambda$-truncated cardinal normal form of $\a$.
\end{itemize}
\end{definition}

For every ordinal $\a$, there is a unique pair $(\b,\g)$ such that $\a = \card{\a} \cdot \b + \g$ and $\g<\card{\a}$, and it follows that every ordinal has a unique cardinal normal form. In particular, the cardinal normal form of $\a$ is definable from $\a$.
Similarly, given some cardinal $\lambda$, every ordinal has a unique $\lambda$-truncated cardinal normal form.
Note that the $\lambda$-truncated cardinal normal form of $\a$ can be obtained by collapsing all the $<\!\lambda$-sized terms of the cardinal normal form of $\a$.
This implies that the $\lambda$-truncated cardinal normal form of $\a$ is definable from $\a$ even without having $\lambda$ as a parameter, although it may not be uniformly definable for all $\a$ unless we include $\lambda$ as a parameter in the definition. 
By convention, the empty sum is the cardinal normal form of $0$.

\begin{definition}\label{def:daleth}
Fix an infinite regular cardinal $\lambda$, and let $\a$ be an ordinal.
\begin{itemize}
\item[$\circ$] The \emph{depth of $\a$}, denoted $\daleth(\a)$, is the number of terms in the $\lambda$-truncated cardinal normal form of $\a$.
\item[$\circ$] If $\dlt_0+\dots+\dlt_{n-1}+\dlt_n$ is the $\lambda$-truncated cardinal normal form of $\a$, we refer to the $\dlt_j$ as the \emph{normal terms} of $\a$, and we write $\dlt_j = \dlt_j(\a)$ for all $j < \daleth(\a) = n+1$.
\item[$\circ$] If $j \leq \daleth(\a)$, then the $j^\mathrm{th}$ \emph{normal segment} of $\a$ is 
$$\lfloor \a \rfloor_j = \dlt_0(\a)+\dots+\dlt_{j-1}(\a).$$
\item[$\circ$] If $j < \daleth(\a)$, then the $j^\mathrm{th}$ \emph{normal interval} of $\a$ is 
$$\mathcal I_j(\a) = [\lfloor \a \rfloor_j,\lfloor \a \rfloor_{j+1}).$$
\end{itemize}
\end{definition}

Note that all the terms in this definition really depend on both $\a$ and $\lambda$. In what follows, when dealing with a long $\lambda$-approximation sequence, the terms above are always defined from $\lambda$, never any other cardinal.

\begin{lemma}\label{lem:a}
Let $\seq{M_\a}{\a < \mu}$ be a long $\lambda$-approximation sequence, and fix $\a < \mu$. Let $j < \daleth(\a)$ and let $f(\b) = \lfloor \a \rfloor_j + \b$ for all $\b < \dlt_j(\a)$. Then $\seq{M_{f(\b)}}{\b < \dlt_j(\a)}$ is a long $\lambda$-approximation sequence.
\end{lemma}
\begin{proof}
It suffices to check property $(2)$ of Definition~\ref{def:approx}, as everything else in the definition is inherited from the sequence $\seq{M_\b}{\b < \mu}$. If $\b < \dlt_j(\a)$, then $\seq{M_\xi}{\xi < f(\b)} \in M_{f(\b)}$ because $\seq{M_\b}{\b < \mu}$ is a long $\lambda$-approximation sequence. But $\lfloor \a \rfloor_j$ and $\b$ are definable from $f(\b)$ via its cardinal normal form. As $f(\b) \in M_{f(\b)}$, this gives $\lfloor \a \rfloor_j,\b \in M_{f(\b)}$ and hence $f \in M_{f(\b)}$. Having $f, \b, \seq{M_\xi}{\xi < f(\b)} \in M_{f(\b)}$ gives $\seq{M_{f(\xi)}}{\xi < \b} \in M_{f(\b)}$. 
\end{proof}

\begin{lemma}\label{lem:b}
Let $\seq{M_\a}{\a < \mu}$ be a long $\lambda$-approximation sequence, and fix some
$\a = \card{\a} \cdot (\z+1) \leq \mu$.
For each $\b < \card{\a}$, let $N_\b = M_{\card{\a} \cdot \z + \b}$.
Then
\begin{enumerate}
\item for every $M \in \set{M_\b}{\b < \a}$, there is some $N \in \set{N_\b}{\b < \card{\a}}$ with $M \sub N$, and 
\item $\seq{N_\b}{\b < \card{\a}}$ is a long $\lambda$-approximation sequence.
\end{enumerate}
\end{lemma}
\begin{proof}
The cases $\a = 0$ and $\z = 0$ are both trivial, so suppose $\a,\z \geq 1$.

Let $\g = \card{\a} \cdot \z$. Fix a surjection $f: \card{\a} \to \g$ such that $f \in M_\g$, and fix $\b < \card{\a}$. 
The pair $(\g,\b)$ is definable from the ordinal $\g+\b$ via its cardinal normal form, because $|\b| < |\g|$; therefore $\g,\b \in M_{\g+\b} = N_\b$. Now $\g \in M_{\g+\b}$ implies $M_\g \sub M_{\g+\b}$, which implies $f \in M_{\g+\b}$. 
Because $f,\b \in M_{\g+\b}$, we have $f(\b) \in M_{\g+\b}$, and therefore $M_{f(\b)} \sub M_{\g+\b} = N_\b$. 
As $f$ is surjective, $(1)$ follows.
For $(2)$, note that $\g,\b,\seq{M_\xi}{\xi < \g+\b} \in M_{\g+\b}$ implies $\seq{M_{\g+\xi}}{\xi < \b} \in M_{\g+\b}$, which is to say $\seq{N_\xi}{\xi < \b} \in N_\b$.
The remaining parts of Definition~\ref{def:approx} are easy to verify, so $\seq{N_\b}{\b < \card{\a}}$ is a long $\lambda$-approximation sequence.
\end{proof}

\begin{lemma}\label{lem:c}
Let $\seq{M_\a}{\a < \mu}$ be a long $\lambda$-approximation sequence, and let $i < \daleth(\mu)$. Then $\set{M_\a}{\a \in \mathcal I_i(\mu)}$ is $\aleph_0$-directed. In particular, if $\daleth(\mu) = 1$ then $\set{M_\a}{\a < \mu}$ is $\aleph_0$-directed.
\end{lemma}
\begin{proof}
This is proved as Lemma 2.4 in \cite{Dave2}.
\end{proof}

\begin{lemma}\label{lem:kisdefinable}
Let $\lambda$ and $\mu$ be infinite cardinals, with $\lambda$ regular, and let $\seq{M_\a}{\a < \mu}$ be a long $\lambda$-approximation sequence. If $\lambda$ is a successor, then $\lambda \in M_\a$ for all $\a > 0$. If $\lambda$ is weakly inaccessible, then $\lambda \in M_\a$ for all $\a \geq \lambda$.
\end{lemma}
\begin{proof}
If $\lambda$ is a successor, say $\lambda = \k^+$, the fact that $M_0$ is weakly $\lambda$-closed implies $\card{M_0} = \k$. Thus $\lambda$ is definable from $M_0$. As $M_0 \in M_\a$ for all $\a > 0$, this implies $\lambda \in M_\a$.
If $\lambda$ is weakly inaccessible, then $\card{M_\xi} < \lambda$ for all $\xi < \mu$. However, $\card{M_\xi} \geq \card{\xi}$ for all $\xi < \lambda$, because $\xi \in M_\xi$ and, as $M_\xi$ is weakly $\lambda$-closed, $\xi \sub M_\xi$. Hence $\lambda = \sup \set{\card{M_\xi}}{\xi < \a}$ whenever $\a \geq \lambda$, and this defines $\lambda$ from $\seq{M_\xi}{\xi < \a}$ in $M_\a$.
\end{proof}

\begin{lemma}\label{lem:d}
Assume \gch. Let $\lambda$ and $\mu$ be infinite cardinals, with $\lambda$ regular, and let $\seq{M_\a}{\a < \mu}$ be a long $\lambda$-approximation sequence. Then $\set{M_\a}{\a < \mu}$ is $\min\{ \cf(\mu),\lambda\}$-directed.
\end{lemma}
\begin{proof}
Let us fix $\lambda$ and proceed by induction on $\mu$. If $\mu \leq \lambda$, then $M_\b \sub M_\a$ for all $\b < \a < \mu$, because $M_\a$ is weakly $<\!\lambda$-closed and $\seq{M_\b}{\b < \a} \in M_\a$. This implies $\set{M_\a}{\a < \mu}$ is $\min\{ \cf(\mu),\lambda\}$-directed.
So assume $\mu > \lambda$. 

If $\mu$ is a limit cardinal, there is an increasing sequence $\seq{\nu_\xi}{\xi < \cf(\mu)}$ of regular cardinals in $[\lambda,\mu)$ with limit $\mu$. By the inductive hypothesis, $\set{M_\a}{\a < \nu_\xi}$ is $\lambda$-directed for each $\xi < \cf(\mu)$. Thus $\set{M_\a}{\a < \mu}$ is the union of an increasing, length-$\cf(\mu)$ chain of $\lambda$-directed sets. It follows that it is $\min\{ \cf(\mu),\lambda\}$-directed.

Finally, suppose $\mu = \nu^+$ for some cardinal $\nu \geq \lambda$. This implies $\lambda = \min\{\cf(\mu),\lambda\}$, so we wish to show that $\set{M_\xi}{\xi < \mu}$ is $\lambda$-directed. Note that $\nu \in M_\nu$ implies $\mu = \nu^+ \in M_\nu$. Furthermore, $\lambda \in M_\nu$ by Lemma~\ref{lem:kisdefinable}, and $\mu = \mu^{<\lambda}$ by \gch. Thus there is a surjection $f : \mu \to [\mu]^{<\lambda}$, and (by elementarity) there is some such $f$ in $M_\nu$. 
Let $\a < \mu$ and $\b = \sup f(\a)$.
As $\mu$ is a cardinal, $\daleth(\mu) = 1$ and so it follows from Lemma~\ref{lem:c} that $\set{M_\a}{\a < \mu}$ is $\aleph_0$-directed.
Therefore there is some $\g < \mu$ such that $\a,M_\b,f \in M_\g$.
Then $f(\a) \in M_\g$ and as $M_\g$ is weakly $\lambda$-closed, this implies $f(\a) \sub M_\g$. 
Also, $M_\b \in M_\g$ implies $\b < \g$, so $\xi < \g$ for every $\xi \in f(\a)$. As $\seq{M_\xi}{\xi < \g} \in M_\g$, this (along with $f(\a) \sub M_\g$) implies $M_\xi \in M_\g$ for every $\xi \in f(\a)$. As $M_\g$ is weakly $\lambda$-closed, this implies $M_\xi \sub M_\g$ for every $\xi \in f(\a)$. Given our choice of $f$, this shows the union of every $<\!\lambda$-sized subset of $\set{M_\a}{\a < \mu}$ is contained in some member of $\set{M_\a}{\a < \mu}$.
\end{proof}

By Lemmas \ref{lem:facts}$(5)$, \ref{lem:sage}, and \ref{lem:d}, we will be done if we can show that every closed long $\k^+$-approximation sequence has the $\k$-sage Davies property. To show this, we will use the following family of combinatorial principles.

\begin{definition}
Let $\lambda \geq \eta \geq \k$ be infinite cardinals with $\eta$ and $\k$ regular. The Very Weak Square principle for $(\lambda,\eta,\k)$, denoted $\vws(\lambda,\eta,\k)$, is the statement that there is a sequence $\seq{C_\a}{\a < \lambda^+}$ and a club $D \sub \lambda^+$ such that, for each $\a \in D$ with $\cf(\a) \geq \eta$,
\begin{itemize}
\item[$\circ$] $C_\a$ is a cofinal subset of $\a \cap D$ with order type $\cf(\a)$, and
\item[$\circ$] for each $X \in [C_\a]^{<\k}$, there is some $\xi < \a$ such that $X \sub C_\xi \in [C_\a]^{<\k}$.
\end{itemize}
\end{definition}

Let us note the following facts concerning $\vws(\lambda,\eta,\k)$:
\begin{itemize}
\item[$\circ$] $\vws(\lambda,\eta,\k)$ implies $\vws(\lambda,\eta',\k)$ whenever $\eta \leq \eta'$.
\item[$\circ$] $\lambda^{<\k} = \lambda$ implies $\vws(\lambda,\k,\k)$: set $E = \set{\a < \lambda^+}{\cf(\a) \geq \k}$ and find an enumeration $\seq{C_\a}{\a \in \lambda^+ \setminus E}$ of $(\lambda^+)^{<\k}$ such that $C_\a \sub \a$ for all $\a \in \lambda^+ \setminus E$, and then, for each $\a \in E$, take $C_\a$ to be any unbounded subset of $\a$ with order type $\cf(\a)$; finally, take $D$ to be the closure in $\lambda^+$ of the set of all $\a < \lambda^+$ satisfying $[\a]^{<\k}\subset\set{C_\b}{\b<\a}$.
\item[$\circ$] If \gch holds then $\set{\a < \lambda^+}{\cf(\a) \geq \k} \in I[\lambda^+]$ implies $\vws(\lambda,\k,\k)$, where $I[\lambda^+]$ denotes the approachability ideal on $\lambda^+$ (cf. \cite{Shelah,Dzamonja&Shelah,Foreman&Magidor}). We note that $\set{\a < \lambda^+}{\cf(\a) \geq \k} \in I[\lambda^+]$ follows from $\square^*_\lambda$. In particular, $\gch+\forall \k\leq\eta\leq\lambda \, \vws(\lambda,\eta,\k)$ is consistent (assuming \zfc is); it follows, for example, from $\VL$ \cite{Jensen}.
\item[$\circ$] The principle $\vws(\lambda,\eta,\k)$ generalizes the Very Weak Square Principle of Foreman and Magidor \cite{Foreman&Magidor}: their principle at $\lambda$, in our notation, is $\vws(\lambda,\aleph_1,\aleph_1)$. Foreman and Magidor show that $\vws(\lambda,\aleph_1,\aleph_1)$ is strictly weaker than $\gch+\square^*_\lambda$, and (unlike $\square^*_\lambda$) $\forall\lambda\,\vws(\lambda,\aleph_1,\aleph_1)$ is consistent with \gch plus the existence of very large cardinals, such as supercompact and huge cardinals.
\end{itemize}

Foreman and Magidor showed that
$\mathsf{GCH}+\neg\vws(\aleph_\omega,\aleph_n,\aleph_1)$
is consistent relative to large cardinals if $n=1$, but the
cases $n=2,3,4,\ldots$ are open problems.
Results from pcf theory motivate the following question:
\begin{question}
Does $\mathsf{GCH}$ imply $\vws(\aleph_\omega,\aleph_4,\aleph_1)$?
\end{question}
We note that Foreman and Magidor asked a similar question in \cite{Foreman&Magidor}, namely whether $\mathsf{GCH}$ implies
a principle similar to $\vws(\aleph_\omega,\aleph_2,\aleph_1)$.

We are ready at last to prove the main technical lemma showing that (under the right hypotheses) the simple construction from Lemma~\ref{lem:facts}$(5)$ produces $\k$-sage Davies trees.

\begin{lemma}\label{lem:square}
Assume \gch and assume that $\vws(\lambda,\eta,\k)$ holds for every singular $\lambda \in [\eta,\mu)$ with $\cf(\lambda) < \k$, where $\k$, $\eta$, and $\mu$ are infinite regular cardinals with $\k \leq \eta$. Then every closed long $\k^+$-approximation sequence of length $\mu$ has the $(\eta,\k)$-Davies property.
\end{lemma}
\begin{proof}
We will prove that for every closed long $\k^+$-approximation sequence $\seq{M_\xi}{\xi < \a}$, there is a set $\mathcal N$ such that
\begin{itemize}
\item[$(i)$] each $N \in \mathcal N$ is the union of a $\k$-directed subset of $\set{M_\xi}{\xi < \a}$,
\item[$(ii)$] $\card{\mathcal N} < \eta$, and
\item[$(iii)$] $\bigcup_{\xi < \a}M_\xi = \bigcup \mathcal N.$
\end{itemize}
This implies the conclusion of the lemma, because if $\seq{M_\a}{\a < \mu}$ is any $\mu$-length closed long $\k^+$-approximation sequence, then every initial segment $\seq{M_\xi}{\xi < \a}$ is also a closed long $\k^+$-approximation sequence and therefore admits a set $\mathcal N_\a$ with the three properties described above. This is precisely the definition of the $(\eta,\k)$-Davies property for $\seq{M_\a}{\a < \mu}$.

The proof that every $\seq{M_\xi}{\xi < \a}$ admits an $\mathcal N$ as described above proceeds by induction on $\a$.

For the base case (and the first $\eta$ cases after that): if $\a < \eta$, then we may take $\mathcal N = \set{M_\xi}{\xi < \a}$. This trivially meets our requirements.

If $\a = \k = \eta$, then we may take $\mathcal N = \left\{ \bigcup_{\xi < \a}M_\xi \right\}$. This trivially satisfies $(ii)$ and $(iii)$, and it satisfies $(i)$ by Lemma~\ref{lem:d}.

In all remaining cases, $\a > \k$. This is assumed for the remainder of the proof. At this point, our argument breaks into several cases.
In what follows, $\daleth(\a)$ is defined using $\lambda = \k^+$ in Definition~\ref{def:daleth}.

\vspace{3mm}
\noindent \emph{Case 1:} Suppose $\daleth(\a) \geq 2$. 
\vspace{1mm}

For each $j < \daleth(\a)$, define $N^j_\b = M_{\lfloor \a \rfloor_j + \b}$ for every $\b < \dlt_j(\a)$. By Lemma~\ref{lem:a}, $\seq{N^j_\b}{\b < \dlt_j(\a)}$ is a closed long $\k^+$-approximation sequence for each $j < \daleth(\a)$. Furthermore, $\daleth(\a) \geq 2$ implies $\dlt_j(\a) < \a$ for each $j < \daleth(\a)$. Therefore, by induction, for each sequence $\seq{N^j_\b}{\b < \dlt_j(\a)}$ there is a set $\mathcal N^j$ satisfying the three properties listed above. But then $\mathcal N = \bigcup_{j < \daleth(\a)} \mathcal N^j$ meets our requirements. Properties $(i)$ and $(ii)$ are automatic, because they hold for the $\mathcal N^j$, and $(iii)$ follows from the fact that $\seq{M_\xi}{\xi < \a}$ is a concatenation of the sequences $\seq{N^j_\b}{\b < \dlt_j(\a)}$, so that $\bigcup_{\xi < \a} M_\xi = \bigcup_{j < \daleth(\a)}\bigcup_{\b < \dlt_j(\a)}N^j_\b = \bigcup_{j < \daleth(\a)}\bigcup \mathcal N^j = \bigcup \mathcal N$.

\vspace{3mm}
\noindent \emph{Case 2:} Suppose $\cf(\a) < \eta$.
\vspace{1mm}

In this case, let $C$ be an unbounded subset of $\a$ with $|C| = \cf(\a)$. For each $\g \in C$, there is, by the inductive hypothesis, some $\mathcal N^\g$ satisfying the three requirements described above for $\seq{M_\xi}{\xi < \g}$. Let $\mathcal N = \bigcup_{\g \in C} \mathcal N^\g$. Then $\mathcal N$ satisfies $(i)$ because each $\mathcal N^\g$ does. Also, $\mathcal N$ satisfies $(ii)$ because each $\mathcal N^\g$ does and $|C| < \eta$. Finally, $\mathcal N$ satisfies $(iii)$ because $\bigcup_{\xi < \a} M_\xi = \bigcup_{\g \in C} \bigcup_{\xi < \g}M_\xi = \bigcup_{\g \in C}\mathcal N^\g = \bigcup \mathcal N$.

\vspace{3mm}
\noindent \emph{Case 3:} Suppose $\daleth(\a) = 1$ and $\cf(\a) \geq \eta$.
\vspace{1mm}

Because $\daleth(\a) = 1$, there is a unique ordinal $\b > 0$ such that $\a = |\a| \cdot \b$. Note that because $\cf(\a) \geq \eta$, either $\b$ is a successor ordinal or else $\cf(\b) \geq \eta$. This case breaks into four sub-cases, depending on what $\b$ is.

\vspace{3mm}
\noindent \emph{Case 3a:} Suppose $\b = 1$.
\vspace{1mm}

In this case, take $\mathcal N = \left\{ \bigcup_{\xi < \a}M_\xi \right\}$. This trivially satisfies $(ii)$ and $(iii)$. Lemma~\ref{lem:d} implies $\set{M_\xi}{\xi < \a}$ is $\min\{\cf(\a),\k^+\}$-directed, and in this case $\cf(\a) \geq \eta \geq \k$; thus $(i)$ holds.

\vspace{3mm}
\noindent \emph{Case 3b:} Suppose $\b = \g+1$ for some $\g \geq 1$.
\vspace{1mm}

In this case, let $N_\z = M_{|\a| \cdot \g + \z}$ for each $\z < |\a|$, and let $\seq{N_\z}{\z < |\a|}$. By Lemma~\ref{lem:b}, $\seq{N_\z}{\z < |\a|}$ is a long $\k^+$-approximation sequence with the property that $\bigcup \set{N_\z}{\z < |\a|} = \bigcup \set{M_\xi}{\xi < \a}$. As $|\a| < \a$, the inductive hypothesis implies there is some $\mathcal N$ satisfying our requirements for $\seq{N_\z}{\z < |\a|}$. But then it is clear that $\mathcal N$ also satisfies our requirements for $\seq{M_\xi}{\xi < \a}$.

\vspace{3mm}
\noindent \emph{Case 3c:} Suppose $\b$ is a limit ordinal and $\cf(|\a|) \geq \k$.
\vspace{1mm}

In this case, as in case 3a, take $\mathcal N = \left\{ \bigcup_{\xi < \a}M_\xi \right\}$. This trivially satisfies $(ii)$ and $(iii)$.
Towards showing that $(i)$ holds, suppose $J \in [\a]^{<\k}$; we will and find $\dlt < \a$ such that $\bigcup_{\xi \in J}M_\xi \sub M_\dlt$.
Note that $\cf(\a) \geq \eta$ and $\a = |\a| \cdot \b$ implies $\cf(\b) \geq \eta \geq \k$.
Therefore, we may choose $\g < \b$ such that $\sup(J) < |\a| \cdot \g$.
From $\g$ define $\seq{N_\z}{\z < |\a|}$ as in case 3b.
By Lemma~\ref{lem:b}, $\seq{N_\z}{\z < |\a|}$ is a long $\k^+$-approximation sequence such that for every $\xi < |\a| \cdot \g$ we may choose $h(\xi) < |\a|$ such that $M_\xi \sub N_{h(\xi)}$.
By Lemma~\ref{lem:d}, $\set{N_\z}{\z < |\a|}$ is $\min\{\cf(|\a|),\k^+\}$-directed, and in this case $\cf(|\a|) \geq \k$.
Therefore, there is some $\z < |\a|$ such that $\bigcup_{\xi \in J} N_{h(\xi)} \sub N_\z$; hence, $\bigcup_{\xi \in J} M_\xi \sub M_\dlt$ where $\dlt = |\a| \cdot \g + \z$.
Thus, $(i)$ holds.

\vspace{3mm}
\noindent \emph{Case 3d:} Suppose $\b$ is a limit ordinal and $\cf(|\a|) < \k$.
\vspace{1mm}

This is the last and most difficult case, and it is where we use our Very Weak Square hypothesis. Note that, as in case 3c, $\cf(\b) \geq \eta$.

Let $\lambda = |\a|$ and $\chi = \cf(\lambda)$. Observe that
$$\lambda^+ > \a > \lambda > \eta \geq \k > \chi.$$
(The first, fourth, and fifth inequalities follow directly from our assumptions and definitions, while the second and third hold, respectively, because $\b > 1$, and because $\lambda \geq \eta$ and $\lambda$ is singular while $\eta$ is regular.) 

Each of $\lambda$, $\eta$, and $\k$ is in $M_{\lambda+\eta+\k}$, because each is definable via the cardinal normal form of $\lambda+\eta+\k$. Applying $\vws(\lambda,\eta,\k)$ inside $M_{\lambda+\eta+\k}$, there is
a sequence $\seq{C_\dlt}{\dlt < \lambda^+}$ and a club $D \sub \lambda^+$, with $D,\seq{C_\dlt}{\dlt < \lambda^+} \in M_{\lambda+\eta+\k}$, such that, for each $\dlt \in D$ with $\cf(\dlt) \geq \eta$,
\begin{itemize}
\item[$\circ$] $C_\dlt$ is a cofinal subset of $\dlt \cap D$ with order type $\cf(\dlt)$, and
\item[$\circ$] for each $X \in [C_\dlt]^{<\k}$, there is some $\xi < \dlt$ such that $X \sub C_\xi \in [C_\dlt]^{<\k}$.
\end{itemize}
By thinning out $D$ if necessary, we may (and do) assume that if $\dlt \in D$ then $(\dlt,\dlt+\lambda) \cap D = \0$. We also assume, for convenience, that $0 \in D$.

Define $f: \lambda^+ \to D$ so that $f(\g) = \text{the } \g^{\mathrm{th}} \text{ member of }D$.
Fix a function $g \in M_{\lambda+\eta+\k}$ such that $g: \lambda^+ \times \lambda \to \lambda$, and for each $\g < \lambda^+$, the function $g(\g,\cdot)$ is a bijection $\lambda \to [f(\g),f(\g+1))$.

Note that $C_{f(\b)}$ has order type $\cf(\b) \geq \eta$, and therefore any $\cf(\b)$-sized subset of $C_{f(\b)}$ is unbounded in $C_{f(\b)}$.

Let $\seq{\nu_\xi}{\xi < \chi}$ be a strictly increasing sequence of regular cardinals such that $\sup_{\xi < \chi} \nu_\xi = \lambda$. 
Because $C_{f(\b)} \sub f(\b) = g[\b \times \lambda] = \bigcup_{\xi < \chi}g[\b \times \nu_\xi]$ and $\chi < \cf(\b)$, there is some $\rho < \chi$ and some $B \sub C_{f(\b)}$ such that $B \sub g[\b \times \nu_\rho]$ and $B$ is unbounded in $C_{f(\b)}$. 
By thinning out $B$ if necessary, we may (and do) assume that for each $\g < \b$ there is at most one $\z < \nu_\rho$ for which $g(\g,\z) \in B$.
Let
$$A = \set{\g < \b}{g(\g,\z) \in B \text{ for some } \z < \nu_\rho}.$$
Because $B$ is unbounded in $C_{f(\b)}$, $A$ is unbounded in $\b$.

Choose a function $c: \lambda^+ \times \lambda \to [\lambda]^{\leq \lambda}$ in $M_{\lambda+\eta+\k}$ such that for all $\dlt < \lambda^+$,
$$\set{C_\xi}{\xi < \dlt} = \set{c(\dlt,\g)}{\g < \lambda}.$$ 

Temporarily fix $J \in [B]^{<\k}$. 
By assumption, there is some $\z < f(\b)$ with $J \sub C_\z \in [C_{f(\b)}]^{<\k}$. 
Let $\dlt_J$ denote the least ordinal in $D$ satisfying $\dlt_J > \z$.
For each $\dlt \in B$ with $\dlt \geq \dlt_J$, there is, by our choice of $\seq{C_\xi}{\xi < \lambda^+}$ and $c$, some $\g < \lambda$ such that $J \sub c(\dlt,\g)$ and $|c(\dlt,\g)| < \k$. 
Furthermore, because $\mathrm{ordertype}(B) = \cf(\b) \geq \eta > \chi$, there is some unbounded set of $\dlt$'s that share the same bound on $\g$: more formally, there is some $\xi_J < \chi$ such that for unboundedly many $\dlt \in B$, there is some $\g < \nu_{\xi_J}$ such that $J \sub c(\dlt,\g)$ and $|c(\dlt,\g)| < \k$. 

Now unfix $J \in [B]^{<\k}$. The previous paragraph shows there is a function $a: [B]^{<\k} \to \chi$ such that for any $J \in [B]^{<\k}$, there are unboundedly many $\dlt \in B$ with the property that some $\g < \nu_{a(J)}$ satisfies $J \sub c(\dlt,\g)$ and $|c(\dlt,\g)| < \k$.

The set $[B]^{<\k}$ is $\k$-directed, and therefore any partition of $[B]^{<\k}$ into fewer than $\k$ pieces must contain a single piece that is cofinal in $[B]^{<\k}$. 
(Here, to say $\mathcal D \sub [B]^{<\k}$ is ``cofinal'' means that for each $J \in [B]^{<\k}$ there is some $X \in \mathcal D$ such that $J \sub X$.)
In particular, the fibers of $a$ form a partition of $[B]^{<\k}$, and so one of these fibers must be cofinal in $[B]^{<\k}$. Thus there is some $\t < \chi$ such that for each $J \in [B]^{<\k}$ there are unboundedly many $\dlt \in B$ with the property that some $\g < \nu_\t$ satisfies $J \sub c(\dlt,\g)$ and $|c(\dlt,\g)| < \k$. Because increasing $\t$ does not change this property, we may (and do) assume that $\t \geq \rho$ and $\nu_\t > \eta$.

For each $\xi \in [\t,\chi)$,  let
$$N_\xi = \textstyle \bigcup \set{M_{\lambda \cdot \dlt + \g}}{(\dlt,\g) \in A \times \nu_\xi}$$
and let $\mathcal N = \set{N_\xi}{\xi \in [\t,\chi)}$. It is clear that $\card{\mathcal N} = \chi < \eta$, so $(ii)$ holds. 

To see that $(iii)$ holds, first note that by Lemma~\ref{lem:b}, if $\lambda \cdot \dlt < \a$ then
$$\textstyle \bigcup \set{M_\g}{\g < \lambda \cdot (\dlt+1)} = \bigcup \set{M_{\lambda \cdot \dlt + \g}}{\g < \lambda}.$$
Now $\lambda \cdot \dlt < \a$ if and only if $\dlt < \b$. This, together with the fact that $A$ is an unbounded subset of $\b$, implies
$$\textstyle \bigcup_{\dlt \in A} \set{M_{\lambda \cdot \dlt + \g}}{\g < \lambda} = \bigcup_{\g < \a} M_\g.$$
Because $\sup_{\xi < \chi}\nu_\xi = \lambda$,
\begin{align*}
\textstyle \bigcup_{\dlt \in A} \set{M_{\lambda \cdot \dlt + \g}}{\g < \lambda} &= \textstyle \bigcup \set{M_{\lambda \cdot \dlt + \g}}{\dlt \in A \text{ and } \g < \lambda} \\
&= \textstyle \bigcup_{\xi \in [\t,\chi)} \set{M_{\lambda \cdot \dlt + \g}}{(\dlt,\g) \in A \times \nu_\xi} = \bigcup \mathcal N,
\end{align*}
and thus $(iii)$ holds.

It remains to prove that $(i)$ holds. For this, it suffices to show that $\set{M_{\lambda \cdot \dlt + \g}}{(\dlt,\g) \in A \times \nu_\xi}$ is $\k$-directed for each $\xi \in [\t,\chi)$.

Fix $\xi \in [\t,\chi)$.
It suffices to show that for every $I \in [A \times \nu_\xi]^{<\k}$, there is some $(\dlt,\g) \in A \times \nu_\xi$ such that $I \sub (\dlt \times \nu_\xi) \cap M_{\lambda \cdot \dlt + \g}$, because then for every $(\dlt',\g') \in I$ we have $\lambda \cdot \dlt' + \g' < \lambda \cdot \dlt + \g$ and $\lambda \cdot \dlt'+\g' \in M_{\lambda \cdot \dlt + \g}$, and this implies $M_{\lambda \cdot \dlt' + \g'} \sub M_{\lambda \cdot \dlt + \g}$.

Fix $I \in [A \times \nu_\xi]^{<\k}$, and fix some $H \in [A]^{<\k}$ and $K \in [\nu_\xi]^{<\k}$ such that $I \sub H \times K$.
Recall that $\k \in M_\lambda$ by Lemma~\ref{lem:kisdefinable}, and choose a function $\psi: \lambda \to [\lambda]^{<\k}$ in $M_\lambda$ such that for every regular cardinal $\e \in [\k,\lambda)$, $\psi$ maps $\e$ onto a cofinal subset of $[\e]^{<\k}$. (\gch implies such a function exists, and then elementarity implies there is such a function in $M_\lambda$.)
Fix some $\z < \nu_\xi$ such that $K \sub \psi(\z)$. 

For every $i \in [\lambda,\lambda^+)$, $M_i$ contains $\lambda$ (because $\lambda = |i|$ is definable from $i \in M_i$); this implies $M_\lambda \sub M_i$ and therefore $\psi \in M_i$. Thus if $\z \in M_i$, then $\psi(\z) \in M_i$ and (because $M_i$ is weakly $<\!\k^+$-closed) $\psi(\z) \sub M_i$. Now suppose $(\dlt,\g) \in A \times \nu_\xi$ such that $H \sub \dlt$ and $H \cup \{\z\} \sub M_{\lambda \cdot \dlt + \g}$. Then
$$I \sub H \times K \sub H \times \psi(\z) \sub M_{\lambda \cdot \dlt + \g},$$
and this implies $I \sub (\dlt \times \nu_\xi) \cap M_{\lambda \cdot \dlt + \g}$.
Hence, to complete the proof that $(i)$ holds for $\mathcal N$,
it suffices to find $(\dlt,\g) \in A \times \nu_\xi$ such that $H \sub \dlt$ and $H \cup \{\z\} \sub M_{\lambda \cdot \dlt + \g}$.

Let $J = B \cap g[H \times \nu_\rho]$. Because of how we thinned out $B$, for each $x \in H$ there is exactly one $y \in \nu_\rho$ such that $g(x,y) \in J$. Therefore $|J| < \k$.
By our choice of $\nu_\t$, there are unboundedly many $\e \in B$ with the property that some $\pi<\nu_\t$ satisfies $J \sub c(\e,\pi)$ and $|c(\e,\pi)| < \k$.
In particular, we may pick some such $\e$ (and $\pi$) with $\e$ large enough so that if $(\dlt,\varsigma) = g^{-1}(\e)$, then $H \sub \dlt$ and $\dlt \geq 2$.
So that $\e$ is defined precisely, let us take $\e$ to be the least such ordinal.
Note that $(\dlt,\varsigma) = g^{-1}(\e)$ implies $\dlt \in A$ and $\varsigma < \nu_\rho$.

Thus we have obtained our $\dlt \in A$ with $H \sub \dlt$; it remains to find $\g < \nu_\xi$ such that $H \cup \{\z\} \sub M_{\lambda \cdot \dlt + \g}$. 
Let
$$G = \set{x < \lambda^+}{\text{for some } y < \lambda,\, g(x,y) \in c(\e,\pi) }.$$
By our choice of $\e$ and $\pi$, $H \sub G$. Because $g$ is injective, $|G| \leq |c(\e,\pi)| < \k$. Because each $M_i$ is weakly $<\!\k$-closed, this means that if $G,\z \in M_i$ then $G \cup \{\z\} \sub M_i$.

The set $G$ is defined above from the parameters $\lambda$, $c$, $\e$, $\pi$, and $g$. 
Also $\e = g(\dlt,\varsigma)$, and thus $\e$ is definable from parameters $g$, $\dlt$, and $\varsigma$. 

Now, $\g < \lambda$ and $\dlt > 0$ implies $\lambda,\dlt \in M_{\lambda \cdot \dlt + \g}$, because $\lambda \cdot \dlt + \g \in M_{\lambda \cdot \dlt + \g}$ and both $\lambda$ and $\dlt$ are definable from $\lambda \cdot \dlt + \g$ (e.g., via its cardinal normal form).
If $\eta,\k \in M_{\lambda \cdot \dlt + \g}$ then $\lambda + \eta + \k \in M_{\lambda \cdot \dlt + \g}$ and (because $\dlt \geq 2$) $\lambda + \eta + \k < \lambda \cdot \dlt + \g$, which implies $c,g \in M_{\lambda+\eta+\k} \sub M_{\lambda \cdot \dlt + \g}$.

Thus, to find some $\g < \nu_\xi$ such that $H \cup \{\z\} \sub M_{\lambda \cdot \dlt + \g}$ and finish the proof, it suffices to find $\g < \nu_\xi$ such that $P = \{ \eta,\k,\pi,\varsigma,\z \} \sub M_{\lambda \cdot \dlt + \g}$.
For each $j < \nu_\xi$, define $N_j = M_{\lambda \cdot \dlt + j}$. By Lemma~\ref{lem:a}, $\seq{N_j}{j < \nu_\xi}$ is a long $\k^+$-approximation sequence. Because $\nu_\xi$ is a cardinal (which implies $\daleth(\nu_\xi) = 1$), Lemma~\ref{lem:c} implies $\set{N_j}{j < \nu_\xi}$ is $\aleph_0$-directed. 
But notice that each $j \in P$ is $<\! \nu_\xi$. For each $j \in P$, we have $\lambda \cdot \dlt + j \in N_j$ and therefore $j \in N_j$. Because $\set{N_j}{j < \nu_\xi}$ is $\aleph_0$-directed, there is some $\g < \nu_\xi$ such that $j \in N_\g$ for every $j \in P$. This $\g$ is as required.
\end{proof}

Let $\vws$ abbreviate the statement that $\vws(\lambda,\eta,\k)$ holds for all infinite cardinals $\lambda \geq \eta \geq \k$ with $\eta,\k$ regular. 

\begin{theorem}\label{thm:trees}
Assume $\gch+\vws$, and let $\k,\mu$ be infinite regular cardinals with $\k < \mu$. Then for any set $p$, there is a $\k$-sage Davies tree for $\mu$ over $p$.
\end{theorem}
\begin{proof}
This follows from Lemmas \ref{lem:facts}$(5)$, \ref{lem:sage}, \ref{lem:d}, and \ref{lem:square}.
\end{proof}

In fact, our proof only every uses that \gch holds below $\mu$, so this slightly weaker hypothesis suffices for Theorem~\ref{thm:trees}. If there are no singular cardinals below $\mu$, then the $\vws$ hypothesis becomes superfluous:

\begin{corollary}\label{cor:smalltrees}
Let $\k,\mu$ be infinite regular cardinals with $\k < \mu < \aleph_\w$ and suppose \gch holds below $\mu$. Then for any set $p$, there is a $\k$-sage Davies tree for $\mu$ over $p$.
\end{corollary}

\begin{theorem}\label{thm:main}
$\gch+\vws$ implies \axiom. Consequently, $\gch+\vws$ implies that for any quasi-regular space $X$, if \emph{\small NONEMPTY} has a winning strategy in $\mathrm{BM}(X)$, then then \emph{\small NONEMPTY} has a winning $2$-tactic.
\end{theorem}
\begin{proof}
This follows from Theorem~\ref{thm:trees} and Corollaries \ref{cor:trees} and \ref{cor:teeitup}.
\end{proof}

\begin{corollary}\label{cor:small}
Suppose $\PP$ is a separative poset with $|\PP| < \aleph_\w$, and that $\gch$ holds below $\card{\PP}$. Then \axiom$\!\!(\PP)$ holds.
\end{corollary}

By modifying the arguments in Section~\ref{sec:BM}, this corollary implies that if \gch holds up to some $\aleph_n$, then any quasi-regular space witnessing Telg\'arsky's conjecture must have $\pi$-weight at least $\aleph_n$.

\section{The independence of \axiom}\label{sec:independence}

In this section we show that \axiom is independent of \zfc.

Recall that for $f,g \in \w^\w$, $f \leq^* g$ means that $f(n) \leq g(n)$ for all but finitely many values of $n$.
A subset $A$ of $\w^\w$ is \emph{unbounded} if there is no $g \in \w^\w$ such that $f \leq^* g$ for all $f \in A$, and
$A$ is a \emph{dominating family} if for all $f \in \w^\w$, there is some $g \in A$ such that $f \leq^* g$.
The smallest size of an unbounded subset of $\w^\w$ is denoted by $\bdd$, and the smallest size of a dominating family is denoted by $\dom$.

\begin{theorem}\label{thm:Hechler}
Let $\HH$ denote the Hechler forcing. Then $\HH$ is ccc, and if $\bdd > \aleph_1$ then $\pnt(\HH) \geq \bdd$. 
Consequently, $\axiom(\mathrm{ccc}+\mathrm{separative})$ implies $\bdd = \aleph_1$.
\end{theorem}
\begin{proof}
Recall that the Hechler forcing is 
$$\HH = \set{(s,f)}{s \in \w^{<\w} \text{ and } f \in \w^\w}$$
with the extension relation defined by having $(t,g) \leq (s,f)$ if and only if
\begin{itemize}
\item[$\circ$] $t$ extends $s$,
\item[$\circ$] $g(n) \geq f(n)$ for all $n \geq \domain(t)$, and 
\item[$\circ$] $t(n) \geq f(n)$ for all $n \in \domain(t) \setminus \domain(s)$.
\end{itemize}
It is well-known (and not difficult to see) that $\HH$ has the ccc. 
So to prove the theorem, we must show $\pnt(\HH) \geq \bdd$ whenever $\bdd > \aleph_1$.

Let $\DD$ be a dense sub-poset of $\HH$. 

We claim there is some $s \in \w^{<\w}$ such that $\DD_{s} = \set{f \in \w^\w}{(s,f) \in \DD}$ is a dominating family. Aiming for a contradiction, suppose $\DD_{s}$ is not a dominating family for any $s \in \w^{<\w}$. Then for every $s \in \w^{<\w}$, there is some $f_s \in \w^\w$ such that $f_s \not\leq^* g$ for all $g \in \DD_{s}$. Let $f \in \w^\w$ be any function such that $f_s \leq^* f$ for all $s \in \w^{<\w}$. (Such a function must exist because $\w^{<\w}$ is countable.) For all $s \in \w^{<\w}$, we have $f \not\leq^* g$ for all $g \in \DD_s$. Thus no member of $\DD$ extends $(\0,f)$. This is a contradiction, because $\DD$ is dense in $\HH$.

Fix $s \in \w^{<\w}$ such that $\DD_{s} = \set{f \in \w^\w}{(s,f) \in \DD}$ is a dominating family, and (using the hypothesis $\bdd > \aleph_1$) fix some uncountable $\k < \bdd$, and fix $X \sub \DD_s$ with $\card{X} = \k$. (This is possible because $\card{\DD_s} \geq \dom \geq \bdd > \k$.) Because $\card{X} < \bdd$, there is some $g \in \w^\w$ such that $f \leq^* g$ for every $f \in X$; because $\DD_s$ is a dominating family, there is some $h \in \DD_s$ with $g \leq^* h$. Thus $f \leq^* h \in \DD_s$ for every $f \in X$.

Given $t \in \w^{<\w}$, let $s \cat t$ denote, as usual, the member of $\w^{<\w}$ such that $s \cat t \rest \domain(s) = s$, and $s \cat t(\domain(s)+i) = t(i)$ for all $i \in \domain(t)$. Note that $h \geq^* f$ implies that $(s \cat t,h) \leq (s,f)$ for some $t \in \w^{<\w}$. Therefore, by the pigeonhole principle, there is some $t \in \w^{<\w}$ such that $(s \cat t,h) \leq (s,f)$ for $\k$-many $f \in X$. Because $\DD$ is dense in $\HH$, there is some $p \in \DD$ with $p \leq (s \cat t,h)$, and so $p \leq (s,f)$ for $\k$-many $f \in X$. Thus $\pnt(\DD) > \k$. As $\DD$ was an arbitrary dense sub-poset of $\HH$, and as $\k$ was any cardinal $<\!\bdd$, this shows $\pnt(\HH) \geq \bdd$.
\end{proof}

Let us point out that there are two natural ways to strengthen the statement of Theorem~\ref{thm:Hechler}, and both of them are false:

\begin{observation}
It is not necessarily true that $\pnt(\HH) = \bdd$ when $\bdd > \aleph_1$. In other words, the inequality of Theorem~\ref{thm:Hechler} can be strict. 
\end{observation}
\begin{proof}
To see this, let us first recall a theorem of Hechler \cite{Hechler}: If $P$ is a partially ordered set in which every countable subset has an upper bound, then there is a ccc forcing extension in which $P$ is isomorphic to a cofinal subset of $(\w^\w,\leq^*)$. Let $\k$ and $\lambda$ be uncountable regular cardinals with $\k < \lambda$, and let $P = \k \times \lambda$ with the natural product order, i.e., $(\b_0,\dlt_0) \leq (\b_1,\dlt_1)$ if and only if $\b_0 \leq \b_1$ and $\dlt_0 \leq \dlt_1$. By Hechler's Theorem, there is some ccc forcing extension in which there is a mapping $(\b,\dlt) \mapsto f_{\b,\dlt}$ from $\k \times \lambda$ onto a cofinal subset of $(\w^\w,\leq^*)$. It is not too difficult to show this implies $\bdd = \k$ and $\dom = \lambda$.
 Now, suppose $\DD$ is any dense sub-poset of $\HH$. As in the proof of Theorem~\ref{thm:Hechler}, there is some $s \in \w^{<\w}$ such that $\DD_s$ is a dominating family. For each $f \in \DD_s$, let $\pi(f)$ denote some pair $(\b,\dlt)$ such that $f \leq^* f_{\b,\dlt}$. Because $\DD_s$ is a dominating family, $\mathrm{Im}(\pi) = \set{\pi(f)}{f \in \DD}$ is cofinal in $\k \times \lambda$. In particular, $\card{\mathrm{Im}(\pi)} = \lambda$, and by the pigeonhole principle (using $\k < \lambda$ and the regularity of $\lambda$) there is some particular $\b$ such that the first coordinate of $\pi(f)$ is equal to $\b$ for $\lambda$-many members of $\mathrm{Im}(\pi)$. Using  the fact that $\k < \lambda$, there is some $\dlt$ such that of the $\lambda$-many members of $\mathrm{Im}(\pi)$ with first coordinate $\b$, at least $\k$ of them have second coordinate $< \! \dlt$. This implies that $\card{\set{f \in \DD}{f \leq^* f_{\b,\dlt}}} \geq \k$. Arguing as in the last paragraph of the proof of Theorem~\ref{thm:Hechler}, it follows that  for some $t \in \w^{<\w}$, $(s \cat t,f_{\b,\dlt}) \leq (s \cat t,f)$ for at least $\k$-many $f \in \DD_s$. But this implies $\pnt(\DD) > \k$, and as $\DD$ was arbitrary, $\pnt(\HH) > \k = \bdd$.
\end{proof}

\begin{observation}
It is consistent that $\bdd = \aleph_1$ and $\pnt(\HH) = \aleph_0$.
\end{observation}
\begin{proof}
To see this, suppose $\bdd = \dom = \aleph_1$. Then there is a sequence $\seq{f_\a}{\a < \w_1}$ of members of $\w^\w$ with the following three properties: 
\begin{itemize}
\item[$\circ$] $\set{f_\a}{\a < \w_1}$ is a dominating family, 
\item[$\circ$] $f_\a \leq^* f_\b$ whenever $\a < \b$, and
\item[$\circ$] for any $\a < \w_1$ and $k \in \w$, $\set{\b < \w_1}{f_\b(n) \leq f_\a(n) \text{ for all } n \geq k}$ is finite.
\end{itemize}
We leave the details of constructing such a sequence to the reader, but point out that it is very similar to Hausdorff's construction of indestructible $(\w_1,\w_1)$-gaps in $\w^\w$ \cite{Hausdorff}. Given such a sequence, let 
$$\DD = \set{(s,f_\a)}{s \in \w^{<\w},\a < \w_1} \sub \HH.$$ 
Because $\set{f_\a}{\a < \w_1}$ is a dominating family, $\DD$ is a dense sub-poset of $\HH$. Now let $(s,f) \in \HH$. Pick $\a < \w_1$ such that $f \leq^* f_\a$ and pick $k \geq \domain(s)$ such that $f(n) \leq f_\a(n)$ for all $n \geq k$. Fix $t \sub s$. If $(s,f) \leq (t,f_\b)$, then we must have $f_\b(n) \leq f(n)$ for all $n \geq \domain(s)$ and hence $f_\b(n) \leq f_\a(n)$ for all $n \geq k$. By our third condition listed above, this implies that $(s,f) \leq (t,f_\b)$ for only finitely many $\b < \w_1$. As this holds for each of the finitely many $t \sub s$, we may conclude that $(s,f)$ extends only finitely many conditions in $\DD$. Hence $\pnt(\DD) = \aleph_0$, and hence $\pnt(\HH) = \aleph_0$.
\end{proof}

Finally, we list three other examples one may use to show the failure of $\axiom$ under various set-theoretic hypotheses:
\begin{itemize}
\item[$\circ$] Let $\BB$ denote the random real forcing. Then $\BB$ is ccc, and if \ma holds then $\pnt(\BB) > \aleph_1$. (The proof essentially follows from Exercise 27 in chapter 2 of \cite{Kunen}.)
\item[$\circ$] Suppose there is a $P_\k$-point $\U \in \omega^*$ for some $\k > \aleph_1$. (The existence of such an ultrafilter follows, for example, from $\continuum = \pseudo > \aleph_1$.) Let $\mathbb M(\U)$ denote the Mathias-type forcing for diagonalizing $\U$, namely 
$$\mathbb M(\U) = \set{(s,A)}{A \in \U}$$ 
where $(t,B) \leq (s,A)$ if and only if $t \supseteq s$, $B \hspace{-.4mm}\setminus\hspace{-.4mm} \max t \sub A  \hspace{-.4mm}\setminus\hspace{-.4mm} \max t$, and $t \setminus s \sub A \setminus \max s$.
It is easy to show that $\mathbb M(\U)$ is ccc, and one may show, by an argument similar to the proof of Theorem~\ref{thm:Hechler}, that $\pnt(\mathbb M(\U)) \geq \k$. 
\item[$\circ$] More generally, one does not really need $\U$ to be an ultrafilter in the preceding example: it is enough to have $\U$ be any filter generated by a $\k$-directed subset of $\mathcal P(\omega)$ (that is, $\k$-directed with respect to the $\sub^*$ relation). The existence of such a family follows from the existence of $P_\k$-points in $\w^*$, or from the inequality $\bdd > \aleph_1$ (just use the sets of points in $\omega \times \omega$ below the graphs of a dominating family of functions). But the existence of such a filter is not a theorem of \zfc: Kunen showed in his thesis that there is no such filter in the Cohen model. 
\end{itemize}


\end{document}